\documentclass[12pt,a4paper]{amsart}
\usepackage{amssymb,eucal}
\usepackage{hyperref}

\calclayout

\newcommand{\+}{\protect\nobreakdash-}
\renewcommand{\:}{\colon}

\newcommand{\rarrow}{\longrightarrow}

\newcommand{\ot}{\otimes}
\newcommand{\ocn}{\odot}
\newcommand{\tim}{\rightthreetimes}

\newcommand{\lrarrow}{\mskip.5\thinmuskip\relbar\joinrel\relbar\joinrel
 \rightarrow\mskip.5\thinmuskip\relax}

\newcommand{\bu}{{\text{\smaller\smaller$\scriptstyle\bullet$}}}

\DeclareMathOperator{\Hom}{Hom}
\DeclareMathOperator{\Ext}{Ext}

\DeclareMathOperator{\Tot}{Tot}

\newcommand{\Modl}{{\operatorname{\mathsf{--Mod}}}}
\newcommand{\Modr}{{\operatorname{\mathsf{Mod--}}}}
\newcommand{\Discr}{{\operatorname{\mathsf{Discr--}}}}
\newcommand{\Contra}{{\operatorname{\mathsf{--Contra}}}}

\newcommand{\Sets}{\mathsf{Sets}}
\newcommand{\Fil}{\mathsf{Fil}}
\newcommand{\Com}{\mathsf{Com}}
\newcommand{\Hot}{\mathsf{Hot}}
\newcommand{\Ac}{\mathsf{Ac}}
\newcommand{\Ab}{\mathsf{Ab}}

\newcommand{\proj}{\mathsf{proj}}
\renewcommand{\flat}{\mathsf{flat}}
\renewcommand{\cot}{\mathsf{cot}}
\newcommand{\bctr}{\mathsf{bctr}}
\newcommand{\bco}{\mathsf{bco}}

\newcommand{\id}{{\mathrm{id}}}

\newcommand{\N}{\mathcal N}

\newcommand{\R}{\mathfrak R}
\newcommand{\fP}{\mathfrak P}
\newcommand{\fQ}{\mathfrak Q}
\newcommand{\fF}{\mathfrak F}
\newcommand{\fG}{\mathfrak G}
\newcommand{\fH}{\mathfrak H}
\newcommand{\fI}{\mathfrak I}
\newcommand{\fJ}{\mathfrak J}
\newcommand{\fM}{\mathfrak M}

\newcommand{\fA}{\mathfrak A}
\newcommand{\fB}{\mathfrak B}
\newcommand{\fC}{\mathfrak C}
\newcommand{\fU}{\mathfrak U}
\newcommand{\fS}{\mathfrak S}
\newcommand{\fT}{\mathfrak T}

\newcommand{\sA}{\mathsf A}
\newcommand{\sB}{\mathsf B} 
\newcommand{\sC}{\mathsf C}
\newcommand{\sD}{\mathsf D}
\newcommand{\sE}{\mathsf E}
\newcommand{\sF}{\mathsf F}

\newcommand{\sH}{\mathsf H}
\newcommand{\sK}{\mathsf K}
\newcommand{\sL}{\mathsf L}

\newcommand{\sS}{\mathsf S}
\newcommand{\sT}{\mathsf T}
\newcommand{\sX}{\mathsf X}

\newcommand{\boZ}{\mathbb Z}

\newcommand{\Section}[1]{\bigskip\section{#1}\medskip}
\theoremstyle{plain}
\newtheorem{thm}{Theorem}[section]

\newtheorem{lem}[thm]{Lemma}
\newtheorem{prop}[thm]{Proposition}
\newtheorem{cor}[thm]{Corollary}
\theoremstyle{definition}

\newtheorem{rem}[thm]{Remark}

\newtheorem{qst}[thm]{Question}

\begin{document}

\title{A contramodule generalization of Neeman's \\
flat and projective module theorem}

\author{Leonid Positselski}

\address{Institute of Mathematics, Czech Academy of Sciences \\
\v Zitn\'a~25, 115~67 Praha~1 \\ Czech Republic}

\email{positselski@math.cas.cz}

\begin{abstract}
 This paper builds on top of~\cite{Pflcc}.
 We consider a complete, separated topological ring $\R$ with
a countable base of neighborhoods of zero consisting of
open two-sided ideals.
 The main result is that the homotopy category of projective left
$\R$\+contramodules is equivalent to the derived category of the exact
category of flat left $\R$\+contramodules, and also to the homotopy
category of flat cotorsion left $\R$\+contramodules.
 In other words, a complex of flat $\R$\+contramodules is contraacyclic
(in the sense of Becker) if and only if it is an acyclic complex with
flat $\R$\+contramodules of cocycles, and if and only if it is coacyclic
as a complex in the exact category of flat $\R$\+contramodules.
 These are contramodule generalizations of theorems of Neeman and of
Bazzoni, Cort\'es-Izurdiaga, and Estrada.
\end{abstract}

\maketitle

\tableofcontents

\section*{Introduction}
\medskip

 The \emph{flat and projective periodicity theorem} of Benson and
Goodearl~\cite[Theorem~2.5]{BG} claims that, in the context of
modules over an arbitrary ring $R$, if $0\rarrow F\rarrow P\rarrow F
\rarrow0$ is a short exact sequence with a flat module $F$ and
a projective module $P$, then the module $F$ is actually projective.
 This result was rediscovered by Neeman~\cite[Remark~2.15]{Neem},
who proved the following much stronger claims.

 Let $F^\bu$ be a complex of flat left $R$\+modules.
 Then the following three conditions are equivalent:
\begin{enumerate}
\renewcommand{\theenumi}{\roman{enumi}}
\item for every complex of projective left $R$\+modules $P^\bu$,
any morphism of complexes of $R$\+modules $P^\bu\rarrow F^\bu$ is
homotopic to zero;
\item $F^\bu$ is an acyclic complex of $R$\+modules with flat
$R$\+modules of cocycles;
\item $F^\bu$ is a directed colimit of contractible complexes of
projective left $R$\+modules.
\end{enumerate}
 See~\cite[Theorem~8.6]{Neem}.

 A further development came in the paper by Bazzoni, Cort\'es-Izurdiaga,
and Estrada~\cite{BCE}, where the \emph{cotorsion periodicity
theorem}~\cite[Theorem~1.2(2), Proposition~4.8(2), or
Theorem~5.1(2)]{BCE} was proved.
 It claims that if $0\rarrow M\rarrow C\rarrow M\rarrow0$ is
a short exact sequence of $R$\+modules and the $R$\+module $C$ is
cotorsion, then the $R$\+module $M$ is cotorsion, too.
 As a corollary, it was essentially shown in~\cite[Theorem~5.3]{BCE}
that the following two conditions on a complex of flat left $R$\+modules
$F^\bu$ are equivalent to~(i\+-iii):
\begin{enumerate}
\renewcommand{\theenumi}{\roman{enumi}}
\setcounter{enumi}{3}
\item for every complex of cotorsion left $R$\+modules $C^\bu$,
any morphism of complexes of $R$\+modules $F^\bu\rarrow C^\bu$ is
homotopic to zero;
\item for every complex of flat cotorsion left $R$\+modules $G^\bu$,
any morphism of complexes of $R$\+modules $F^\bu\rarrow G^\bu$ is
homotopic to zero.
\end{enumerate}

 The aim of this paper is to prove the following generalization of
the theorems of Neeman and of Bazzoni, Cort\'es-Izurdiaga, and Estrada.
 Let $\R$ be a complete, separated topological ring with a countable
base of neighborhoods of zero consisting of open two-sided ideals.
 Let $\fF^\bu$ be a complex of flat left $\R$\+contramodules.
 Then the following six conditions are equivalent:
\begin{enumerate}
\renewcommand{\theenumi}{\roman{enumi}${}^{\mathrm{c}}$}
\item for every complex of projective left $\R$\+contramodules
$\fP^\bu$, any morphism of complexes of $\R$\+contramodules
$\fP^\bu\rarrow\fF^\bu$ is homotopic to zero;
\item $\fF^\bu$ is an acyclic complex of $\R$\+contramodules with flat
$\R$\+contramodules of cocycles;
\item $\fF^\bu$ can be obtained from contractible complexes of flat
$\R$\+contramodules using extensions and directed colimits;
\item $\fF^\bu$ is an $\aleph_1$\+directed colimit of total complexes
of short exact sequences of complexes of countably presentable flat
$\R$\+contramodules;
\item for every complex of cotorsion left $\R$\+contramodules $\fC^\bu$,
any morphism of complexes of $\R$\+contramodules $\fF^\bu\rarrow\fC^\bu$
is homotopic to zero;
\item for every complex of flat cotorsion left $\R$\+contramodules
$\fG^\bu$, any morphism of complexes of $\R$\+contramodules $\fF^\bu
\rarrow\fG^\bu$ is homotopic to zero.
\end{enumerate}

 The equivalence of the three conditions (ii${}^{\mathrm{c}}$)
$\Longleftrightarrow$ (iii${}^{\mathrm{c}}$) $\Longleftrightarrow$
(iv${}^{\mathrm{c}}$) was established in
the paper~\cite[Theorem~13.2]{Pflcc}.
 The aim of the present paper is to prove the equivalence of
the four conditions (i${}^{\mathrm{c}}$) $\Longleftrightarrow$
(ii${}^{\mathrm{c}}$) $\Longleftrightarrow$ (v${}^{\mathrm{c}}$)
$\Longleftrightarrow$ (vi${}^{\mathrm{c}}$).

 The study of arbitrary (unbounded) complexes of projective modules
goes back to J\o rgensen's paper~\cite{Jor}.
 The terminology ``contraderived category'' was introduced by
the present author in~\cite{Psemi,Pkoszul}, inspired by Keller's
terminology ``coderived category'' in~\cite{Kel2} (see also~\cite{KLN}).
 The contemporary point of view on the topic came with Becker's
paper~\cite[Proposition~1.3.6]{Bec}.
 Let us spell out some definitions.

 Let $\sB$ be an abelian (or exact) category with enough projective
objects.
 A complex $B^\bu$ in $\sB$ is said to be \emph{contraacyclic}
(in the sense of Becker) if, for every complex of projective objects
$P^\bu$ in $\sB$, any morphism of complexes $P^\bu\rarrow B^\bu$ is
homotopic to zero.
 The triangulated Verdier quotient category $\sD^\bctr(\sB)=
\Hot(\sB)/\Ac^\bctr(\sB)$ of the homotopy category $\Hot(\sB)$ of
complexes in $\sB$ by the thick subcategory of contraaacyclic complexes
$\Ac^\bctr(\sB)\subset\Hot(\sB)$ is called the (\emph{Becker})
\emph{contraderived category} of the abelian/exact category~$\sB$.

 Dually, let $\sA$ be an exact category with enough injective objects.
 A complex $A^\bu$ in $\sA$ is said to be \emph{coacyclic} (in
the sense of Becker) if, for every complex of injective objects $J^\bu$
in $\sA$, any morphism of complexes $A^\bu\rarrow J^\bu$ is homotopic
to zero.
 The quotient category $\sD^\bco(\sA)=\Hot(\sA)/\Ac^\bco(\sA)$ of
the homotopy category $\Hot(\sA)$ by the thick subcategory of coacyclic
complexes $\Ac^\bco(\sA)\subset\Hot(\sA)$ is called the (\emph{Becker})
\emph{coderived category} of the exact category~$\sA$.
 We refer to~\cite[Section~7]{Pksurv} and~\cite[Remark~9.2]{PS4} for
a general discussion of the history and philosophy of the coderived
and contraderived categories.
 In the context of the present paper, the discussion
in~\cite[Section~7]{Pphil} is relevant.

 Generally speaking, for the abelian category $R\Modl$ of modules
over a ring $R$, the coderived category $\sD^\bco(R\Modl)$ is quite
different from the contraderived category $\sD^\bctr(R\Modl)$.
 For the abelian category of $\R$\+contramodules $\R\Contra$,
the construction of the coderived category does not even make sense,
as there are no nonzero injective objects in $\R\Contra$ in general.
 However, the coderived category \emph{of the exact category of
flat $R$\+modules} agrees with the derived and contraderived categories
of the same exact category~\cite[Sections~7.4 and~7.6]{Pphil}.
 In this paper we generalize these results to the exact category
$\R\Contra_\flat$ of flat contramodules over a complete, separated
topological ring with a countable base of neighborhoods of zero
consisting of open two-sided ideals.

 For the abelian category of $R$\+modules $\sB=R\Modl$,
the contraderived category is equivalent to the homotopy category of
complexes of projective objects, $\Hot(R\Modl_\proj)\simeq
\sD^\bctr(R\Modl)$ \,\cite[Proposition~8.1]{Neem}.
 The equivalence of two conditions~(i) and~(ii) above (which is a part
of~\cite[Theorem~8.6]{Neem}) can be restated by saying that a complex of
flat $R$\+modules is contraacyclic if and only if it is acyclic as
a complex in the exact category of flat $R$\+modules $R\Modl_\flat$.
 Besides, it is obvious that a complex in the exact category
$R\Modl_\flat$ is contraacyclic if and only if it is contraacyclic as
a complex in $R\Modl$, because the projective objects of $R\Modl_\flat$
coincide with those of $R\Modl$.
 It follows that the contraderived category of $R$\+modules
is equivalent to the derived category of the exact category of
flat $R$\+modules, $\Hot(R\Modl_\proj)\simeq\sD(R\Modl_\flat)
=\sD^\bctr(R\Modl_\flat)\simeq\sD^\bctr(R\Modl)$.

 On the other hand, the injective objects of the exact category
$R\Modl_\flat$ are the flat cotorsion $R$\+modules.
 So the equivalence of two conditions~(ii) and~(v) above (which follows
immediately from~\cite[Theorem~5.3]{BCE}) can be restated by saying
that the classes of acyclic and coacyclic complexes in the exact
category $R\Modl_\flat$ coincide.
 It also follows from~\cite[Theorem~5.3]{BCE} that the coderived
category of the exact category $R\Modl_\flat$ is equivalent to
the homotopy category of complexes of flat cotorsion $R$\+modules,
$\Hot(R\Modl_\flat^\cot)\simeq\sD^\bco(R\Modl_\flat)$.
 Thus we have $\Hot(R\Modl_\proj)\simeq\sD^\bctr(R\Modl_\flat)
=\sD(R\Modl_\flat)=\sD^\bco(R\Modl_\flat)\simeq\Hot(R\Modl_\flat^\cot)$.

 For the abelian category of $\R$\+contramodules $\sB=\R\Contra$,
the contraderived category is also equivalent to the homotopy category
of complexes of projective objects, $\Hot(\R\Contra_\proj)\simeq
\sD^\bctr(\R\Contra)$.
 This is a quite general result, valid for any locally presentable
abelian category $\sB$ with enough projective objects (which means,
in particular, for the category of left contramodules over any
complete, separated topological ring $\R$ with a base of neighborhoods
of zero formed by open right ideals) \cite[Corollary~7.4]{PS4}.
 When $\R$ has a \emph{countable} base of neighborhoods of zero
consisting of open \emph{two-sided} ideals, the results of
the present paper allow us to say more.

 The equivalence of two conditions (i${}^{\mathrm{c}}$)
and~(ii${}^{\mathrm{c}}$) above (which is claimed in our
Theorem~\ref{main-theorem}) can be restated by saying that a complex
of flat $\R$\+contramodules is contraacyclic if and only if it is
acyclic as a complex in the exact category of flat $\R$\+contramodules
$\R\Contra_\flat$.
 Similarly to the module case, it is obvious that a complex in
the exact category $\R\Contra_\flat$ is contraacyclic if and only if
it is contraacyclic as a complex in $\R\Contra$, because the projective
objects in $\R\Contra_\flat$ and $\R\Contra$ coincide.
 It follows that the contraderived category of $\R$\+contramodules
is equivalent to the derived category of the exact category of
flat contramodules, $\Hot(\R\Contra_\proj)\simeq\sD(\R\Contra_\flat)
=\sD^\bctr(\R\Contra_\flat)\simeq\sD^\bctr(\R\Contra)$.

 The equivalence of two conditions (ii${}^{\mathrm{c}}$)
and~(vi${}^{\mathrm{c}}$) above (which is another assertion of our
Theorem~\ref{main-theorem}) can be restated by saying
that a complex of flat $\R$\+contramodules is coacyclic in
$\R\Contra_\flat$ if and only if it is acyclic in this exact category.
 Furthermore, according to our
Corollary~\ref{becker-coderived-of-flats}, the coderived category of
the exact category $\R\Contra_\flat$ is equivalent to the homotopy
category of complexes of flat cotorsion $\R$\+contramodules,
$\Hot(\R\Contra_\flat^\cot)\simeq\sD^\bco(\R\Contra_\flat)$.
 So we have $\Hot(\R\Contra_\proj)\simeq\sD^\bctr(\R\Contra_\flat)
=\sD(\R\Contra_\flat)=\sD^\bco(\R\Contra_\flat)\simeq
\Hot(\R\Contra_\flat^\cot)$, just as in the module case.

 We refer to the paper~\cite{PS7} for a discussion of another
generalization of the results of Neeman's paper~\cite{Neem}
and the paper of Bazzoni, Cort\'es-Izurdiaga, and Estrada~\cite{BCE},
from modules to curved DG\+modules.
 In the present paper, we discuss contramodule generalizations
of~\cite[Theorem~8.6]{Neem} and~\cite[Theorem~5.3]{BCE}.

 Let us explain the connection with the Benson--Goodearl periodity
theorem.
 Let $0\rarrow F\rarrow P\rarrow F\rarrow0$ be a short exact
sequence of modules with a flat $R$\+module $F$ and a projective
$R$\+module~$P$.
 Splicing copies of this short exact sequence on top of each other
in both the positive and negative cohomological directions, we
obtain an unbounded acyclic complex of projective $R$\+modules $P^\bu$
with flat $R$\+modules of cocycles.
 So the complex $F^\bu=P^\bu$ satisfies~(ii), and consequently
it also satisfies~(i).
 Thus the identity morphism of complexes $P^\bu\rarrow F^\bu$ must
be homotopic to zero.
 Hence the complex $P^\bu$ is contractible, and it follows that
its modules of cocycles are projective.
 This is the argument of~\cite[Remark~2.15]{Neem}.

 Conversely, given an acyclic complex of projective $R$\+modules with
flat $R$\+modules of cocycles, one can chop it up into short exact
sequence pieces, shift them to one and the same cohomological
degree, and take the countable direct sum.
 This produces a short exact sequence of the form $0\rarrow F
\rarrow P\rarrow F\rarrow 0$.
 This argument shows that the Benson--Goodearl periodicity theorem
is essentially equivalent to the assertion that in any acyclic complex
of projective modules with flat modules of cocycles, the modules of
cocycles are actually projective~\cite[proof of Proposition~7.6]{CH},
\cite[Propositions~1 and~2]{EFI}, \cite[Proposition~2.4]{BCE}.

 Similarly, one can start with the equivalence of two conditions
(i${}^{\mathrm{c}}$) and~(ii${}^{\mathrm{c}}$) and deduce the following
version of flat and projective periodicity theorem for contramodules.
 If $0\rarrow\fF\rarrow\fP\rarrow\fF\rarrow0$ is a short exact sequence
of $\R$\+contramodules with a flat contramodule $\fF$ and a projective
contramodule $\fP$, then the contramodule $\fF$ is actually projective.
 A direct argument deducing this assertion from the Benson--Goodearl
periodicity theorem for modules was spelled out
in~\cite[Proposition~12.1]{Pflcc}.
 The equivalence of two conditions (i${}^{\mathrm{c}}$)
and~(ii${}^{\mathrm{c}}$), proved in Theorem~\ref{main-theorem}
in the present paper, is a much stronger result.

 Analogous arguments are applicable to the cotorsion periodicity of
Bazzoni, Cort\'es-Izurdiaga, and Estrada.
 The cotorsion periodicity theorem for $R$\+modules, stated above
in this introduction, is essentially equivalent to the assertion that
in any acyclic complex of cotorsion $R$\+modules, the $R$\+modules
of cocycles are actually cotorsion~~\cite[Theorem~1.2(2),
Proposition~4.8(2), and Theorem~5.1(2)]{BCE}.
 Similarly, the cotorsion periodicity theorem for $\R$\+contramodules
claims that if $0\rarrow\fM\rarrow\fC\rarrow\fM\rarrow0$ is a short
exact sequence of $\R$\+contramodules with a cotorsion contramodule
$\fC$, then $\fM$ is also a cotorsion
contramodule~\cite[Theorem~12.3]{Pflcc}.
 This theorem can be restated by saying that in any acyclic complex
of cotorsion $\R$\+contramodules $\fC^\bu$, the contramodules of cocycles
are cotorsion~\cite[Corollary~12.4]{Pflcc}.

 We refer to the introduction to the paper~\cite{BHP} and
to the preprint~\cite[Sections~7.8 and~7.10]{Pphil} for a further
discussion of periodicity theorems.
 Before we finish this introduction, let us say a few words about
contramodules.

 \emph{Contramodules} are modules with infinite summation operations.
 Contramodule categories, dual-analogous to the categories of comodules,
or torsion modules, or discrete, or smooth modules, can be assigned to
various algebraic structures~\cite{Prev}.
 In most cases, the category of contramodules is abelian.
 In particular, to any complete, separated topological ring $\R$ with
a base of neighborhoods of zero formed by open right ideals, one
assigns the abelian category of left $\R$\+contramodules $\R\Contra$.

 The categories of topological modules are usually \emph{not} abelian;
and indeed, contramodules are \emph{not} topological modules.
 The $\R$\+contramodules are $\R$\+modules endowed with infinite
summation operations with the coefficient families converging to zero
in the topology of~$\R$ \,\cite[Section~2.1]{Prev}, \cite[Sections~1.2
and~5]{PR}, \cite[Section~6]{PS1}.
 There is a naturally induced topology on a left $\R$\+contramodule,
but it need not be even separated.
 So $\R$\+contramodules are nontopological modules over
a topological ring~$\R$, forming an abelian category (in fact,
a locally presentable abelian category with enough projective objects;
cf.~\cite[Sections~6\+-7]{PS4}).

 When the topological ring $\R$ has a countable base of neighborhoods
of zero, the contramodule theory simplifies quite a bit.
 In particular, all left $\R$\+contramodules are complete (though
they still need not be separated) in their induced topologies.
 Furthermore, there is a well-behaved full subcategory of \emph{flat\/
$\R$\+contramodules} $\R\Contra_\flat\subset\R\Contra$.
 Flat left $\R$\+contramodules are separated and complete in their
induced topologies.
 The full subcategory $\R\Contra_\flat$ is resolving and closed under
directed colimits in $\R\Contra$.
 Moreover, the directed colimits are exact functors in
$\R\Contra_\flat$ \,\cite[Sections~6\+-7]{PR} (but not in $\R\Contra$
\,\cite[Examples~4.4]{PR}).

 When $\R$ has a countable base of neighborhoods of zero consisting of
open \emph{two-sided} ideals, flat $\R$\+contramodules can be
described as certain projective systems of flat modules over
the discrete quotient rings of~$\R$.
 The latter fact, however, plays almost no role in the present paper,
where we study complexes of flat $\R$\+contramodules using
$\aleph_1$\+accessibility and $\aleph_1$\+presentability results for
contramodule categories that were obtained in the paper~\cite{Pflcc}
(which was, in turn, based on the paper~\cite{Pacc}).
 The description of flat $\R$\+contramodules in terms of systems of flat
modules over discrete quotient rings of $\R$ was used in~\cite{Pflcc}.

 We also consider the exact category of \emph{cotorsion\/
$\R$\+contramodules} $\R\Contra^\cot$.
 According to~\cite[Theorem~7.17 and Corollary~7.21]{Pphil}, for any
associative ring $R$, the inclusion of exact/abelian categories
$R\Modl^\cot\rarrow R\Modl$ (where $R\Modl^\cot$ is the full subcategory
of cotorsion left $R$\+modules) induces equivalences of their derived
and contraderived categories $\sD(R\Modl^\cot)\simeq\sD(R\Modl)$
and $\sD^\bctr(R\Modl^\cot)\simeq\sD^\bctr(R\Modl)$.
 For a complete, separated topological ring $\R$ with a countable base
of neighborhoods of zero consisting of two-sided ideals, it was shown
in~\cite[Corollary~12.8]{Pflcc} that the inclusion of exact/abelian
categories $\R\Contra^\cot\rarrow\R\Contra$ induces an equivalence of
the derived categories $\sD(\R\Contra^\cot)\simeq\sD(\R\Contra)$.
 In this paper we show that the same inclusion induces an equivalence
of the contraderived categories $\sD^\bctr(\R\Contra^\cot)\simeq
\sD^\bctr(\R\Contra)$.

 Geometrically, flat contramodules over a \emph{commutative} topological
ring with a countable topology base of open ideals can be identified
with flat pro-quasi-coherent pro-sheaves on an ind-affine
$\aleph_0$\+ind-scheme~\cite[Section~7.11.3]{BD2},
\cite[Examples~3.8]{Psemten}, \cite[Example~8.9]{Pflcc}.
 This constitutes a geometric motivation for the present work.

 This paper consists roughly of two parts.
 Sections~\ref{accessible-secn}\+-\ref{countably-presentable-secn}
present the preliminary material, and then in
Sections~\ref{acyclic-are-contraacyclic-secn}\+-%
\ref{contraacyclic-are-acyclic-secn} we prove contramodule versions
of results related to the flat/projective periodicity theorem of
Benson--Goodearl and Neeman.
 Section~\ref{cotorsion-pairs-secn} spells out additional
preliminaries for the second part, and then in
Sections~\ref{complexes-of-cotorsion-secn}\+-%
\ref{contraderived-of-cotorsion-secn} contramodule versions of results
related to the cotorsion periodicity theorem of Bazzoni,
Cort\'es-Izurdiaga, and Estrada are proved.
 The final Section~\ref{five-constructions-secn} summarizes the results
obtained in Sections~\ref{acyclic-are-contraacyclic-secn}\+-%
\ref{contraacyclic-are-acyclic-secn}
and~\ref{complexes-of-cotorsion-secn}\+-%
\ref{contraderived-of-cotorsion-secn}. {\uchyph=0\par}

\subsection*{Acknowledgement}
 I~am grateful to Jan \v St\!'ov\'\i\v cek for helpful discussions.
 The author is supported by the GA\v CR project 23-05148S and
the Czech Academy of Sciences (RVO~67985840).

\Section{Accessible Subcategories} \label{accessible-secn}

 We use the book~\cite{AR} as a general reference source on accessible
and locally presentable categories.
 In particular, we refer to~\cite[Definition~1.4, Theorem~1.5,
Corollary~1.5, Definition~1.13(1), and Remark~1.21]{AR} for a discussion
of \emph{$\kappa$\+directed colimits} (indexed by $\kappa$\+directed
posets) vs.\ \emph{$\kappa$\+filtered colimits} (indexed by
$\kappa$\+filtered small categories).
 Here $\kappa$~is a regular cardinal.

 Let $\sK$ be a category with $\kappa$\+directed (equivalently,
$\kappa$\+filtered) colimits.
 An object $S\in\sK$ is said to be \emph{$\kappa$\+presentable} if
the functor $\Hom_\sK(S,{-})\:\sK\rarrow\Sets$ from $\sK$ to
the category of sets $\Sets$ preserves $\kappa$\+directed colimits.
 When the category $\sK$ is (pre)additive, an object $S\in\sK$ is
$\kappa$\+presentable if and only if the functor
$\Hom_\sK(S,{-})\:\sK\rarrow\Ab$ from $\sK$ to the category of abelian
groups $\Ab$ preserves $\kappa$\+directed colimits.

 The category $\sK$ is said to be
\emph{$\kappa$\+accessible}~\cite[Definition~2.1]{AR} if there exists
a \emph{set} of $\kappa$\+presentable objects $\sS\subset\sK$ such that
every object of $\sK$ is a $\kappa$\+directed colimit of objects
from~$\sS$.
 If this is the case, then the $\kappa$\+presentable objects of $\sK$
are precisely all the retracts (direct summands) of the objects
from~$\sS$.
 The category $\sK$ is said to be \emph{locally
$\kappa$\+presentable}~\cite[Definition~1.17 and Theorem~1.20]{AR}
if $\sK$ is $\kappa$\+accessible and all colimits exists in~$\sK$.

 In the case of the countable cardinal $\kappa=\aleph_0$,
the $\aleph_0$\+presentable objects are known as \emph{finitely
presentable}~\cite[Definition~1.1]{AR}, the $\aleph_0$\+accessible
categories are called \emph{finitely
accessible}~\cite[Remark~2.2(1)]{AR}, and the locally
$\aleph_0$\+presentable categories are known as \emph{locally
finitely presentable}~\cite[Definition~1.9 and Theorem~1.11]{AR}.
 In the case of the cardinal $\kappa=\aleph_1$, we will call
the $\aleph_1$\+presentable objects \emph{countably presentable}.

 Given a category $\sK$ with $\kappa$\+directed colimits and a class
of objects $\sT\subset\sK$, we denote by $\varinjlim_{(\kappa)}\sT
\subset\sK$ the class of all colimits of $\kappa$\+directed diagrams
of objects from $\sT$ in~$\sK$.
 Generally speaking, the class of objects $\varinjlim_{(\kappa)}
\varinjlim_{(\kappa)}\sT$ may differ from $\varinjlim_{(\kappa)}\sT$
\,\cite[Examples~3.5 and~3.8]{PPT}, but under the assumptions of
the following proposition they coincide.

\begin{prop} \label{accessible-subcategory}
 Let\/ $\sK$ be a $\kappa$\+accessible category and\/ $\sT$ be a set
of (some) $\kappa$\+presentable objects in\/~$\sK$.
 Then the full subcategory\/ $\varinjlim_{(\kappa)}\sT\subset\sK$ is
closed under $\kappa$\+directed colimits in\/~$\sK$.
 The category\/ $\sL=\varinjlim_{(\kappa)}\sT$ is $\kappa$\+accessible,
and the $\kappa$\+presentable objects of\/ $\sL$ are precisely all
the retracts of the objects from\/~$\sT$.
 An object $L\in\sK$ belongs to\/ $\varinjlim_{(\kappa)}\sT$ if and only
if, for every $\kappa$\+presentable object $S\in\sK$, every morphism
$S\rarrow L$ in\/ $\sK$ factorizes through an object from\/~$\sT$.
\end{prop}

\begin{proof}
 This well-known result goes back, at least,
to~\cite[Proposition~2.1]{Len}, \cite[Section~4.1]{CB},
and~\cite[Proposition~5.11]{Kra} (notice that what we call finitely
accessible categories were called ``locally finitely presented
categories'' in~\cite{CB,Kra}).
 The nontrivial part is the ``if'' implication in the last assertion;
all the other assertions follow easily.
 See~\cite[Proposition~1.2]{Pacc} for some details.
\end{proof}

 Let $\alpha$~be an ordinal and $\sK$ be a category.
 An \emph{$\alpha$\+indexed chain} (of objects and morphisms) in $\sK$
is a commutative diagram $(K_i\to K_j)_{0\le i<j<\alpha}$ in $\sK$
indexed by the directed poset~$\alpha$.

 Given an additive/abelian category $\sK$, we denote by $\Com(\sK)$
the additive/abelian category of (unbounded) complexes in~$\sK$.

\begin{prop} \label{complexes-presentable-accessible}
 Let $\kappa$ be an uncountable regular cardinal and $\lambda<\kappa$
be a smaller infinite cardinal. \par
\textup{(a)} Let\/ $\sK$ be an additive category with
$\kappa$\+directed colimits.
 Then all complexes of $\kappa$\+presentable objects in\/ $\sK$ are
$\kappa$\+presentable as objects of the category of complexes\/
$\Com(\sK)$. \par
\textup{(b)} Let\/ $\sK$ be a $\kappa$\+accessible additive category
where the colimits of $\lambda$\+indexed chains exist.
 Then the category\/ $\Com(\sK)$ of complexes in\/ $\sK$ is
$\kappa$\+accessible.
 The\/ $\kappa$\+presentable objects of\/ $\Com(\sK)$ are precisely
all the complexes of $\kappa$\+presentable objects in\/~$\sK$.
\end{prop}

\begin{proof}
 Part~(a) holds because $\kappa$\+directed colimits commute with
countable limits in the category $\Sets$ or~$\Ab$,
cf.~\cite[Lemma~1.1]{Pflcc}.
 Part~(b) is a particular case of~\cite[Theorem~6.2]{Pacc};
cf.~\cite[proof of Corollary~10.4]{Pacc}.
\end{proof}

\Section{Contraderived Categories} \label{contraderived-secn}

 We refer to the survey paper~\cite{Bueh} for general background on
exact categories in the sense of Quillen.
 The Yoneda Ext groups in an exact category $\sK$ can be defined, e.~g.,
as $\Ext^n_\sK(X,Y)=\Hom_{\sD(\sK)}(X,Y[n])$ for all $X$, $Y\in\sK$
and $n\ge0$.
 Here $\sD(\sK)$ denotes the derived category of~$\sK$.

 Let $\sK$ be an exact category, and let $\sF$, $\sC\subset\sK$ be
two classes of objects.
 One denotes by $\sF^{\perp_1}\subset\sK$ the class of all objects
$X\in\sK$ such that $\Ext^1_\sK(F,X)=0$ for all $F\in\sF$.
 Dually, the notation ${}^{\perp_1}\sC\subset\sK$ stands for the class
of all objects $Y\in\sK$ such that $\Ext^1_\sK(Y,C)=0$ for all
$C\in\sC$.
 Similarly, $\sF^{\perp_{\ge1}}\subset\sK$ is the class of all objects
$X\in\sK$ such that $\Ext^n_\sK(F,X)=0$ for all $n\ge1$ and $F\in\sF$,
while ${}^{\perp_{\ge1}}\sC\subset\sK$ is the class of all objects
$Y\in\sK$ such that $\Ext^n_\sK(Y,C)=0$ for all $n\ge1$ and $C\in\sC$.

 Given an exact category $\sK$, any full subcategory $\sE\subset\sK$
closed under extensions in $\sK$ can be endowed with an exact category
structure in which the (admissible) short exact sequences in $\sE$ are
the short exact sequences in $\sK$ with the terms belonging to~$\sE$.
 We will say that the exact structure on $\sE$ is \emph{inherited} from
the exact structure on~$\sK$.

 A full subcategory $\sF\subset\sK$ is said to be \emph{self-generating}
if for every admissible epimorphism $K\rarrow F$ in $\sK$ with $F\in\sF$
there exists a morphism $G\rarrow K$ in $\sK$ with $G\in\sF$ such that
the composition $G\rarrow K\rarrow F$ is an admissible epimorphism
$G\rarrow F$ in~$\sK$.
 A full subcategory $\sF\subset\sK$ is said to be \emph{generating}
if for every object $K\in\sK$ there exists an object $G\in\sF$ together
with an admissible epimorphism $G\rarrow K$ in~$\sK$.
 Clearly, any generating full subcategory is self-generating.

 A full subcategory $\sF\subset\sK$ is said to be \emph{self-resolving}
if it is self-generating and closed under extensions and kernels of
admissible epimorphisms.
 We will say that a full subcategory $\sF\subset\sK$ is \emph{resolving}
if it is generating and closed under extensions and kernels of
admissible epimorphisms.
 Clearly, any resolving full subcategory is self-resolving.

\begin{lem} \label{right-orthogonal-to-self-resolving}
 Let\/ $\sK$ be an exact category and\/ $\sF\subset\sK$ be
a self-generating full subcategory closed under kernels of
admissible epimorphisms.
 Then one has\/ $\sF^{\perp_1}=\sF^{\perp_{\ge1}}\subset\sK$.
\end{lem}

\begin{proof}
 The argument from~\cite[Lemma~6.17]{Sto-ICRA} applies.
 See~\cite[Lemma~1.3]{BHP} for some additional details.
\end{proof}

 Let $\sK$ be a category with directed colimits and $\alpha$~be
an ordinal.
 An $\alpha$\+indexed chain $(K_i\to K_j)_{0\le i<j<\alpha}$ in $\sK$
is said to be \emph{smooth} if $K_j=\varinjlim_{i<j}K_i$ for every
limit ordinal $j<\alpha$.

 A smooth chain $(F_i\to F_j)_{0\le i<j<\alpha}$ in an abelian category
$\sK$ with infinite coproducts is said to be an \emph{$\alpha$\+indexed
filtration} (of the object $F=\varinjlim_{i<\alpha}F_i$) if $F_0=0$ and
the morphism $F_i\to F_{i+1}$ is a monomorphism for all
$0\le i<i+1<\alpha$.
 If this is the case, then the object $F=\varinjlim_{i<\alpha} F_i$
is said to be \emph{filtered by} the objects $S_i=F_{i+1}/F_i$,
\,$0\le i<i+1<\alpha$.
 In an alternative terminology, one says that $F$ is
a \emph{transfinitely iterated extension} (\emph{in the sense of
the directed colimit}) of the objects~$S_i$.

 Given a class of objects $\sS\subset\sK$, we denote by
$\Fil(\sS)\subset\sK$ the class of all objects filtered by
(objects isomorphic to) the objects from~$\sS$.
 A class of objects $\sF\subset\sK$ is said to be \emph{deconstructible}
if there exists a \emph{set} of objects $\sS\subset\sK$ such that
$\sF=\Fil(\sS)$.
 The following result is known classically as the \emph{Eklof
lemma}~\cite[Lemma~1]{ET}.

\begin{lem} \label{eklof-lemma}
 Let\/ $\sK$ be an abelian category with infinite coproducts.
 Then, for any class of objects\/ $\sC\subset\sK$, the class of
objects\/ ${}^{\perp_1}\sC\subset\sK$ is closed under transfinitely
iterated extensions in\/~$\sK$.
 In other words, $\Fil({}^{\perp_1}\sC)\subset{}^{\perp_1}\sC$.
\end{lem}

\begin{proof}
 In the stated generality, a proof can be found
in~\cite[Lemma~4.5]{PR}; see also~\cite[Proposition~1.3]{PS4}
and~\cite[Lemma~7.5]{PS6}.
\end{proof}

 The following lemma is well known and easy.

\begin{lem} \label{Ext-from-disk-complex}
 Let\/ $\sB$ be an abelian category, $A\in\sB$ be an object,
and $n\in\boZ$ be an integer.
 Denote by $D_{n,n+1}^\bu(A)$ the contractible two-term complex\/
$\dotsb\rarrow0\rarrow A\overset\id\rarrow A\rarrow0\rarrow\dotsb$
concentrated in the cohomological degrees~$n$ and $n+1$.
 Then, for any complex $B^\bu$ in\/ $\sB$ and all integers $i\ge0$
there is a natural isomorphism of abelian groups
$$
 \Ext^i_{\Com(\sB)}(D_{n,n+1}^\bu(A),B^\bu)
 \simeq\Ext^i_\sB(A,B^n).
$$
\end{lem}

\begin{proof}
 See, e.~g., \cite[Lemma~3.1(5)]{Gil} or~\cite[Lemma~1.8]{Pal}.
\end{proof}

 Given an additive category $\sB$, we denote by $\Hot(\sB)$
the triangulated category of (unbounded) complexes in $\sB$ with
morphisms up to cochain homotopy.
 The notation $B^\bu[i]$, \,$i\in\boZ$ stands for the cohomological
degree shift of a complex~$B^\bu$; so $(B^\bu[i])^n=B^{n+i}$ for
all $n\in\boZ$.
 The following lemma is well known.

\begin{lem} \label{Ext-1-as-homotopy-Hom}
 Let\/ $\sB$ be an abelian category, and let $A^\bu$ and $B^\bu$ be
two complexes in\/~$\sB$.
 Assume that\/ $\Ext^1_\sB(A^n,B^n)=0$ for all $n\in\boZ$.
 Then there is a natural isomorphism of abelian groups
$$
 \Ext^1_{\Com(\sB)}(A^\bu,B^\bu)\simeq\Hom_{\Hot(\sB)}(A^\bu,B^\bu[1]).
$$
\end{lem}

\begin{proof}
 See, e.~g., \cite[Lemma~2.1]{Gil}, \cite[Lemma~1.6]{BHP}
or~\cite[Lemma~1.9]{Pal}.
\end{proof}

 Let $\sB$ be an abelian (or exact) category with enough projective
objects.
 A complex $B^\bu$ in $\sB$ is said to be \emph{contraacyclic}
(\emph{in the sense of Becker}~\cite[Proposition~1.3.6(1)]{Bec}) if,
for every complex of projective objects $P^\bu$ in $\sB$, the complex
of abelian groups $\Hom_\sB^\bu(P^\bu,B^\bu)$ is acyclic.
 The full subcategory of contraacyclic complexes is denoted by
$\Ac^\bctr(\sB)\subset\Hot(\sB)$ or (depending on context)
$\Ac^\bctr(\sB)\subset\Com(\sB)$.
 The triangulated Verdier quotient category
$$
 \sD^\bctr(\sB)=\Hot(\sB)/\Ac^\bctr(\sB)
$$
is called the (\emph{Becker}) \emph{contraderived category} of~$\sB$.

 We are interested in (termwise admissible) short exact sequences 
$0\rarrow K^\bu\rarrow L^\bu\rarrow M^\bu\rarrow0$ of complexes
in~$\sB$.
 Any such short exact sequence can be viewed as a bicomplex with
three rows in~$\sB$.
 Then one can consider its total complex, which we will denote by
$\Tot(K^\bu\to L^\bu\to M^\bu)$.

\begin{lem} \label{absolutely-acyclic-are-contraacyclic}
 Let\/ $\sB$ be an exact category with enough projective objects.
 Then \par
\textup{(a)} for any short exact sequence\/ $0\rarrow K^\bu\rarrow
L^\bu\rarrow M^\bu\rarrow0$ of complexes in\/ $\sB$, the total complex\/
$\Tot(K^\bu\to L^\bu\to M^\bu)$ is contraacyclic; \par
\textup{(b)} assuming that infinite products exists in\/ $\sB$,
the class of all contraacyclic complexes is closed under infinite
products.
\end{lem}

\begin{proof}
 These observations seem to go back to~\cite[Section~3.5]{Pkoszul}
and~\cite[Proposition~4.3]{KLN}.
 Part~(b) is obvious.
 In part~(a), the point is that, for every complex of projective
objects $P^\bu\in\sB$, the short sequence of complexes of abelian
groups $0\rarrow\Hom_\sB^\bu(P^\bu,K^\bu)\rarrow
\Hom_\sB^\bu(P^\bu,L^\bu)\rarrow\Hom_\sB^\bu(P^\bu,M^\bu)\rarrow0$
is exact.
 The total complex of any short exact sequence of complexes of
abelian groups is exact.
 For a generalization, see
Lemma~\ref{hom-from-absolutely-acyclic-in-cotorsion-pair} below.
\end{proof}

\begin{lem} \label{contraacyclic-in-abelian-are-acyclic}
 Let\/ $\sB$ be an \emph{abelian} category with enough projective
objects.
 Then any contraacyclic complex in\/~$\sB$ is acyclic.
\end{lem}

\begin{proof}
 This is the dual assertion to~\cite[Lemma~A.2]{Psemten}
or a particular case of~\cite[Lemma~B.7.3(b)]{Pcosh}.
 It is \emph{not known} whether the assertion of this lemma holds for
exact categories in general; see~\cite[Remark~B.7.4]{Pcosh} for
a discussion.
\end{proof}

 A complex in an exact category $\sE$ is said to be
\emph{absolutely acyclic}~\cite[Section~2.1]{Psemi},
\cite[Section~3.3]{Pkoszul}, \cite[Appendix~A]{Pcosh} if it belongs
to the minimal thick subcategory of $\Hot(\sE)$ containing the total
complexes of (termwise admissible) short exact sequences of complexes
in~$\sE$.
 Equivalently, a complex in $\sE$ is absolutely acyclic if and only if
it belongs to the minimal full subcategory of $\Com(\sE)$ containing
contractible complexes and closed under extensions and direct
summands~\cite[Proposition~8.12]{PS5}.

 It is clear from Lemma~\ref{absolutely-acyclic-are-contraacyclic}(a)
that all absolutely acyclic complexes in $\sB$ are contraacyclic.

 An exact category $\sE$ is said to have \emph{homological
dimension\/~$\le\nobreak d$} (where $d\ge-1$ is an integer) if
$\Ext^{d+1}_\sE(X,Y)=0$ for all $X$, $Y\in\sE$.

\begin{lem} \label{psemi-remark21}
 In an exact category of finite homological dimension, all acyclic
complexes are absolutely acyclic.
\end{lem}

\begin{proof}
 This is~\cite[Remark~2.1]{Psemi}.
 For additional information, see~\cite[Proposition~6.2]{Pflcc}.
\end{proof}

 Let us denote by $\sB_\proj\subset\sB$ the full subcategory of
projective objects in an abelian/exact category~$\sB$.

\begin{thm} \label{becker-contraderived-of-lpacepo}
 For any locally presentable abelian category\/ $\sB$ with enough
projective objects, the inclusion of additive/abelian categories\/
$\sB_\proj\rarrow\sB$ induces a triangulated equivalence
$$
 \Hot(\sB_\proj)\simeq\sD^\bctr(\sB).
$$
\end{thm}

\begin{proof}
 This is~\cite[Corollary~7.4]{PS4}.
 A far-reaching generalization can be found
in~\cite[Corollary~6.14]{PS5}.
 See also Theorem~\ref{contraderived-cotorsion-pair-for-lpacepo} below.
\end{proof}

 Our final lemma in this section is standard and easy, but
helpful to keep in mind.

\begin{lem} \label{pkoszul-lemma16}
 Let\/ $\sH$ be a triangulated category, $\sX\subset\sH$ be a thick
subcategory, and\/ $\sF\subset\sH$ be a full triangulated subcategory.
 Assume that for every object $H\in\sH$ there exists an object
$F\in\sF$ together with a morphism $F\rarrow H$ whose cone belongs
to\/~$\sX$.
 Then the inclusion functor\/ $\sF\rarrow\sH$ induces a triangulated
equivalence of the Verdier quotient categories
$$
 \sF/(\sX\cap\sF)\simeq\sH/\sX.
$$
\end{lem}

\begin{proof}
 This is~\cite[Proposition~10.2.7(ii)]{KS}
or~\cite[Lemma~1.6(a)]{Pkoszul}.
\end{proof}

\Section{Topological Rings} \label{topological-rings-secn}

 The definition of a contramodule over a topological ring goes back
to~\cite[Remark~A.3]{Psemi}.
 We refer to~\cite[Section~2.1]{Prev}, \cite[Sections~1.2 and~5]{PR},
and~\cite[Sections~7\+-8]{Pflcc} for further details and references
in connection with the definitions and assertions presented below
in this section.
 The pedagogical exposition in~\cite[Sections~6\+-7]{PS1} is
particularly recommended.

 Let $A$ be a topological abelian group with a base of neighborhoods
of zero formed by open subgroups.
 The \emph{completion} of $A$ is defined as the abelian group
$\fA=\varprojlim_{U\subset A}A/U$, where $U$ ranges over the open
subgroups of~$A$.
 The abelian group $\fA$ is endowed with the topology of projective
limit of the discrete abelian groups~$A/U$.
 There is a natural morphism of topological abelian groups
$\lambda_A\:A\rarrow\fA$ called the \emph{completion map}.
 The topological abelian group $A$ is called \emph{separated} if
the completion map~$\lambda_A$ is injective, and \emph{complete} if
the map~$\lambda_A$ is surjective.
 The topological abelian group $\fA=\varprojlim_{U\subset A}A/U$
is always separated and complete.

 Given an abelian group $A$ and a set $X$, we denote by
$A[X]=A^{(X)}=\bigoplus_{x\in X}A$ the direct sum of copies of
the group $A$ indexed by the set~$X$.
 The elements of $A[X]$ are interpreted as finite formal linear
combitations $\sum_{x\in X}a_xx$ of elements of $X$ with
the coefficients $a_x\in A$.
 So one has $a_x=0$ for all but a finite subset of indices~$x$.

 Given a complete, separated topological abelian group $\fA$ and
a set $X$, we put $\fA[[X]]=\varprojlim_{\fU\subset\fA}(\fA/\fU)[X]$.
 Here $\fU$ ranges over the all the open subgroups of~$\fA$.
 The elements of $\fA[[X]]$ are interpreted as infinite formal
linear combinations $\sum_{x\in X}a_xx$ of elements of $X$ with
\emph{zero-convergent} families of coefficients $a_x\in\fA$.
 Here the zero-convergence condition means that, for every open
subgroup $\fU\subset\fA$, one has $a_x\in\fU$ for all but a finite
subset of indices $x\in X$.

 For any map of sets $f\:X\rarrow Y$, the induced map $\fA[[f]]\:
\fA[[X]]\rarrow\fA[[Y]]$ takes a formal linear combination
$\sum_{x\in X}a_xx$ to the formal linear combination $\sum_{y\in Y}
b_yy$ with the coefficients defined by the rule
$b_y=\sum_{x\in X}^{f(x)=y}a_x$.
 Here the infinite summation sign in the latter formula denotes
the limit of finite partial sums in the topology of~$\fA$.
 So the assignment $X\longmapsto\fA[[X]]$ is a functor
$\Sets\rarrow\Ab$, but we will mostly consider it as a functor
$\Sets\rarrow\Sets$.

 Now let $\R$ be a complete, separated topological ring with
a base of neighborhoods of zero formed by open right ideals.
 Then the functor $\R[[{-}]]\:\Sets\rarrow\Sets$ has a natural
structure of a \emph{monad} on the category of sets (in the sense
of~\cite[Chapter~VI]{MacLane}).
 For any set $X$, the monad unit $\epsilon_{\R,X}\:X\rarrow\R[[X]]$
is the map taking every element $x\in X$ to the formal linear
combination $\sum_{y\in X}r_yy$ with $r_x=1$ and $r_y=0$ for $y\ne x$.
 The monad multiplication $\phi_{\R,X}\:\R[[\R[[X]]]]\rarrow\R[[X]]$
is the ``opening of parentheses'' map assigning a formal linear
combination to a formal linear combination of formal linear
combinations.
 In order to open the parentheses, one needs to compute pairwise
products of elements of the ring $\R$ and certain infinite sums of
elements of $\R$, which are understood as the limits of finite partial
sums in the topology of~$\R$.
 The conditions imposed on~$\R$ (including, in particular,
the requirement of a topology base of open right ideals) guarantee
the convergence.

 In the general category-theoretic terminology, one usually speaks about
``algebras over monads'', but in our context we prefer to call them
\emph{modules over monads}.
 A \emph{left $\R$\+contramodule} is defined as a module over the monad
$X\longmapsto\R[[X]]$ on the category of sets.
 In other words, a left $\R$\+contramodule $\fP$ is a set endowed with
a \emph{left contraaction map} $\pi_\fP\:\R[[\fP]]\rarrow\fP$ 
satisfying the following \emph{contraassociativity} and
\emph{contraunitality} axioms.
 Firstly, the ``opening of parentheses'' map
$\phi_{\R,\fP}\:\R[[\R[[\fP]]]]\rarrow\R[[\fP]]$ and the map
$\R[[\pi_\fP]]\:\R[[\R[[\fP]]]]\rarrow\R[[\fP]]$ induced by
the map~$\pi_\fP$ must have equal compositions with
the contraaction map~$\pi_\fP$,
$$
 \R[[\R[[\fP]]]]\rightrightarrows\R[[\fP]]\rarrow\fP.
$$
 Secondly, the composition of the map $\epsilon_{\R,\fP}\:\fP\rarrow
\R[[\fP]]$ with the map $\pi_\fP\:\R[[\fP]]\allowbreak\rarrow\fP$
must be equal to the identity map~$\id_\fP$,
$$
 \fP\rarrow\R[[\fP]]\rarrow\fP.
$$

 In particular, given a discrete ring $R$, the similar (but simpler)
construction produces a monad structure on the functor
$X\longmapsto R[X]\:\Sets\rarrow\Sets$.
 Modules over this monad on $\Sets$ are the same things as the usual
left $R$\+modules.
 Now we have two monads on $\Sets$ associated with a topological ring
$\R$ satisfying the conditions above: the monad $X\longmapsto\R[[X]]$
(depending on the topology on~$\R$) and the monad $X\longmapsto\R[X]$
(defined irrespectively of any topology on~$\R$).
 The inclusion of the set of all finite formal linear combinations
into the set of zero-convergent infinite ones is a natural injective
map $\R[X]\rarrow\R[[X]]$.
 Given a left $\R$\+contramodule structure on a set $\fP$,
the composition $\R[\fP]\rarrow\R[[\fP]]\rarrow\fP$ endows $\fP$
with its underlying left $\R$\+module structure.
 In other words, the finite aspects of the contramodule infinite
summation operations define a module structure.

 The category of left $\R$\+contramodules $\R\Contra$ is abelian
with exact functors of infinite products.
 The forgetful functor $\R\Contra\rarrow\R\Modl$ from the category
of left $\R$\+contramodules to the category of left $\R$\+modules
is exact and preserves infinite products.
 We will denote the $\Hom$ and $\Ext$ groups computed in the abelian
category $\R\Contra$ by $\Hom^\R(\fP,\fQ)$ and $\Ext^{\R,*}(\fP,\fQ)$.

 Furthermore, for any set $X$ the contraaction map $\pi_{\R[[X]]}=
\phi_{\R,X}\:\R[[\R[[X]]]]\rarrow\R[[X]]$ defines
a left $\R$\+contramodule structure on the set $\R[[X]]$.
 The resulting left $\R$\+contramodule $\R[[X]]$ is called
the \emph{free} $\R$\+contramodule spanned by the set~$X$.
 For any left $\R$\+contramodule $\fP$, the group of all
$\R$\+contramodule morphisms $\R[[X]]\rarrow\fP$ is naturally bijective
to the group of all maps $X\rarrow\fP$,
$$
 \Hom^\R(\R[[X]],\fP)\simeq\Hom_\Sets(X,\fP).
$$
 It follows easily that there are enough projective objects in
the abelian category $\R\Contra$.
 The projective objects of $\R\Contra$ are precisely all the direct
summands of free $\R$\+contramodules.
 In particular, the free left $\R$\+contramodule with one generator
$\R=\R[[\{*\}]]$ is a projective generator of $\R\Contra$.
 The abelian category $\R\Contra$ is locally $\kappa$\+presentable,
where $\kappa$ is any cardinal greater than the cardinality of
a base of neighborhoods of zero in~$\R$.

 A right $\R$\+module $\N$ is said to be \emph{discrete} if the action
map $\N\times\R\rarrow\N$ is continuous in the given topology of $\R$
and the discrete topology of~$\N$.
 The full subcategory of discrete right $\R$\+modules $\Discr\R\subset
\Modr\R$ is closed under subobjects, quotients, and infinite direct
sums in the category of right $\R$\+modules $\Modr\R$.
 The category $\Discr\R$ is a Grothendieck abelian category.

 The \emph{contratensor product} $\N\ocn_\R\fP$ of a discrete right
$\R$\+module $\N$ and a left $\R$\+contramodule $\fP$ is an abelian
group constructed as the cokernel of (the difference of) the natural
pair of maps
$$
 \N\ot_\boZ\R[[\fP]]\rightrightarrows\N\ot_\boZ\fP.
$$
 In this pair of parallel abelian group maps, the first map
$\N\ot_\boZ\R[[\fP]]\rarrow\N\ot_\boZ\fP$ is the map $\N\ot_\boZ\pi_\fP$
induced by the contraaction map $\pi_\fP\:\R[[\fP]]\rarrow\fP$.
 The second map $\N\ot_\boZ\R[[\fP]]\rarrow\N\ot_\boZ\fP$ is
the composition $\N\ot_\boZ\R[[\fP]]\rarrow\N[\fP]\rarrow\N\ot_\boZ\fP$
of the map $\N\ot_\boZ\R[[\fP]]\rarrow\N[\fP]$ induced by the discrete
right action of $\R$ in $\N$ with a natural surjective map
$\N[\fP]\rarrow\N\ot_\boZ\fP$.
 Here the map $\N\ot_\boZ\R[[\fP]]\rarrow\N[\fP]$ is defined by
the rule $n\ot\sum_{p\in\fP}r_pp\longmapsto\sum_{p\in\fP}(nr_p)p$ for
all $n\in\N$ and $\sum_{p\in\fP}r_pp\in\R[[\fP]]$.
 The sum in the right-hand side is finite, because the family of
elements~$r_p$ is zero-convergent in $\R$, while the right action of
$\R$ in $\N$ is discrete; so $nr_p=0$ in $\N$ for all but a finite
subset of elements $p\in\fP$.
 A natural surjective map $A[B]\rarrow A\ot_\boZ B$ is defined in
the obvious way for any abelian groups $A$ and~$B$.

 The contratensor product functor $\ocn_\R\:\Discr\R\times\R\Contra
\rarrow\Ab$ is right exact and preserves coproducts (in other words,
preserves all colimits) in both of its arguments.
 For any set $X$ and any discrete right $\R$\+module $\N$, there is
a natural isomorphism of abelian groups
$$
 \N\ocn_\R\R[[X]]\simeq\N[X]=\N^{(X)}.
$$
 A left $\R$\+contramodule $\fF$ is called \emph{flat} if
the contratensor product functor ${-}\ocn_\R\nobreak\fF\:\allowbreak
\Discr\R\rarrow\Ab$ is exact.
 It follows that the class of all flat $\R$\+contramodules is preserved
by the directed colimits in $\R\Contra$ (because the directed colimits
are exact in~$\Ab$), and all the projective $\R$\+contramodules
are flat.
 We will denote the full subcategory of flat left $\R$\+contramodules
by $\R\Contra_\flat\subset\R\Contra$.

 Given a closed subgroup $\fA\subset\R$ and a left $\R$\+contramodule
$\fP$, we denote by $\fA\tim\fP\subset\fP$ the image of the composition
$\fA[[\fP]]\rarrow\R[[\fP]]\rarrow\fP$ of the obvious inclusion
$\fA[[\fP]]\hookrightarrow\R[[\fP]]$ with the contraaction map
$\pi_\fP\:\R[[\fP]]\rarrow\fP$.
 For any closed left ideal $\fJ\subset\R$ and any left
$\R$\+contramodule $\fP$, the subgroup $\fJ\tim\fP$ is a subcontramodule
in~$\fP$.
 For any open right ideal $\fI\subset\R$ and any left
$\R$\+contramodule $\fP$, there is a natural isomorphism of abelian
groups
$$
 (\R/\fI)\ocn_\R\fP\simeq\fP/(\fI\tim\fP).
$$
 For any closed right ideal $\fJ\subset\R$ and any set $X$, one has
$$
 \fJ\tim\R[[X]]=\fJ[[X]]\subset\R[[X]].
$$
 In particular, for an open right ideal $\fI\subset\R$, one has
$$
 (\R/\fI)\ocn_\R\R[[X]]\simeq\R[[X]]/\fI[[X]]\simeq(\R/\fI)[X].
$$

 For an open two-sided ideal $\fI\subset\R$ and any left
$\R$\+contramodule $\fP$, the quotient contramodule $\fP/(\fI\tim\fP)$
is a module over the discrete quotient ring $R=\R/\fI$ of~$\R$.
 This is the unique maximal quotient contramodule of $\fP$ whose
left $\R$\+contramodule structure comes from an $R$\+module structure.
 Over a complete, separated topological ring $\R$ with a base of
neighborhoods of zero formed by open \emph{two-sided} ideals,
a left $\R$\+contramodule $\fF$ is flat if and only if the left\/
$\R/\fI$\+module $\fF/(\fI\tim\fF)$ is flat for every open
two-sided ideal $\fI\subset\R$.

\Section{Countably Presentable Contramodules}
\label{countably-presentable-secn}

 The exposition in this section is based on~\cite[Section~6]{PR}
and~\cite[Sections~8\+-10]{Pflcc}.
 In the case of a topological ring with a base of neighborhoods of
zero formed by \emph{two-sided} ideals, \cite[Section~E.1]{Pcosh} is
a relevant reference.

 Let $\R$ be a complete, separated topological ring with a base of
neighborhoods of zero formed by open right ideals.
 Given a left $\R$\+contramodule $\fP$, consider the natural map
$$
 \lambda_{\R,\fP}\:\fP\lrarrow\varprojlim\nolimits_{\fI\subset\R}
 \fP/(\fI\tim\fP).
$$
 Here $\fI$ ranges over all the open right ideals in~$\R$.
 A contramodule $\fP$ is called \emph{separated} if
the map~$\lambda_{\R,\fP}$ is injective (equivalently, this means
that $\bigcap_{\fI\subset\R}(\fI\tim\fP)=0$ in~$\fP$).
 A contramodule $\fP$ is called \emph{complete} if
the map~$\lambda_{\R,\fP}$ is surjective.

\begin{lem}
 Let\/ $\R$ be a complete, separated topological ring with
a \emph{countable} base of neighborhoods of zero consisting of open
right ideals.  Then \par
\textup{(a)} all left\/ $\R$\+contramodules are complete (but
\emph{not} necessarily separated); \par
\textup{(b)} all flat left\/ $\R$\+contramodules are separated.
\end{lem}

\begin{proof}
 This is~\cite[Lemmas~8.1 and~8.3]{Pflcc}.
\end{proof}

\begin{lem} \label{flat-contramodules-resolving}
 Let\/ $\R$ be a complete, separated topological ring with
a \emph{countable} base of neighborhoods of zero consisting of open
right ideals.
 Then the full subcategory of flat left\/ $\R$\+contramodules\/
$\R\Contra_\flat$ is resolving in the abelian category\/ $\R\Contra$.
\end{lem}

\begin{proof}
 This is~\cite[Lemma~8.4(a)]{Pflcc}.
\end{proof}

 In particular, it follows from Lemma~\ref{flat-contramodules-resolving}
that the full subcategory $\R\Contra_\flat$ inherits an exact category
structure from the abelian exact structure on $\R\Contra$.
 We will consider $\R\Contra_\flat$ as an exact category with this
exact category structure.
 Recall that the full subcategory $\R\Contra_\flat$ is preserved by
the directed colimits in $\R\Contra$.

\begin{lem} \label{flat-contramodules-exactness-properties}
 Let\/ $\R$ be a complete, separated topological ring with
a \emph{countable} base of neighborhoods of zero consisting of open
right ideals.
 In this context: \par
\textup{(a)} For any discrete right\/ $\R$\+module\/ $\N$,
the contratensor product functor\/ $\N\ocn_\R\nobreak{-}\,\:
\allowbreak\R\Contra\rarrow\Ab$ restricted to the full subcategory
$\R\Contra_\flat\subset\R\Contra$ is exact on\/ $\R\Contra_\flat$.
 In particular, for any open right ideal\/ $\fI\subset\R$,
the reduction functor\/ $\fF\longmapsto\fF/(\fI\tim\fI)$ is exact
on\/ $\R\Contra_\flat$. \par
\textup{(b)} The directed colimit functors are exact in\/
$\R\Contra_\flat$ (but \emph{not} in\/ $\R\Contra$).
\end{lem}

\begin{proof}
 Part~(a) is~\cite[Lemma~8.4(b)]{Pflcc}.
 Part~(b) is~\cite[Lemma~8.8]{Pflcc}.
\end{proof}

\begin{lem} \label{flat-contramodules-described}
 Let\/ $\R$ be a complete, separated topological ring, and let\/
$\R\supset\fI_1\supset\fI_2\supset\fI_3\supset\dotsb$ be
a \emph{countable} base of neighborhoods of zero\/ in $\R$ consisting
of open \emph{two-sided} ideals\/ $\fI_n$, \,$n\ge1$.
 Denote by $R_n=\R/\fI_n$ the corresponding discrete quotient rings.
 Then the rule assigning to a flat left\/ $\R$\+contramodule\/ $\fF$
the collection of $R_n$\+modules $F_n=\fF/(\fI_n\tim\fF)$ is
an equivalence between the category of flat left\/ $\R$\+contramodules\/
$\R\Contra_\flat$ and the category of the following sets of data:
\begin{enumerate}
\item for every $n\ge1$, a flat left $R_n$\+module $F_n$ is given;
\item for every $n\ge1$, an isomorphism of left $R_n$\+modules
$F_n\simeq R_n\ot_{R_{n+1}}F_{n+1}$ is given.
\end{enumerate}
\end{lem}

\begin{proof}
 This is~\cite[Lemma~8.6(b)]{Pflcc}.
\end{proof}

\begin{prop} \label{presentable-contramodules}
 Let $\kappa$~be a regular cardinal, and let\/ $\R$ be a complete,
separated topological ring with a base of neighborhoods of zero of
cardinality smaller than~$\kappa$, consisting of open right ideals.
 Then the abelian category of left\/ $\R$\+contramodules\/
$\R\Contra$ is locally $\kappa$\+presentable.
 The $\kappa$\+presentable objects of\/ $\R\Contra$ are precisely
all the cokernels of morphisms of free\/ $\R$\+contramodules spanned
by sets of cardinality smaller than~$\kappa$.
\end{prop}

\begin{proof}
 This is~\cite[Proposition~9.1(a)]{Pflcc}.
\end{proof}

 Assuming that the topological ring $\R$ has a countable base of
neighborhoods of zero consisting of open right ideals, we will say that
a left $\R$\+contramodule $\fP$ is \emph{countably generated} if $\fP$
is a quotient contramodule of a free contramodule $\R[[X]]$ spanned by
a countable set~$X$.
 This is equivalent to $\fP$ being an $\aleph_1$\+generated object of
$\R\Contra$ in the sense of~\cite[Definition~1.67]{AR}.

\begin{prop} \label{flat-contramodules-accessible}
 Let\/ $\R$ be a complete, separated topological ring, and let\/
$\R\supset\fI_1\supset\fI_2\supset\fI_3\supset\dotsb$ be
a \emph{countable} base of neighborhoods of zero\/ in $\R$ consisting
of open \emph{two-sided} ideals\/ $\fI_n$, \,$n\ge1$.
 Denote by $R_n=\R/\fI_n$ the corresponding discrete quotient rings.
 Then the category\/ $\R\Contra_\flat$ of flat left\/
$\R$\+contramodules is\/ $\aleph_1$\+accessible.
 Furthermore, \par
\textup{(a)} a flat left\/ $\R$\+contramodule is\/
$\aleph_1$\+presentable as an object of\/ $\R\Contra_\flat$ if and only
if it is\/ $\aleph_1$\+presentable as an object of\/ $\R\Contra$; \par
\textup{(b)} a flat left\/ $\R$\+contramodule\/ $\fF$ is\/
$\aleph_1$\+presentable if and only if the $R_n$\+module
$F_n=\fF/(\fI_n\tim\fF)$ is countably presentable for every $n\ge1$.
\end{prop}

\begin{proof}
 This is~\cite[Lemmas~9.4 and~9.5, and Theorem~10.1]{Pflcc}.
\end{proof}

\begin{prop} \label{projective-dimension-of-flat-countably-presented}
 Let\/ $\R$ be a complete, separated topological ring with
a \emph{countable} base of neighborhoods of zero consisting of open
\emph{two-sided} ideals.
 Then any countably presentable flat\/ $\R$\+contramodule has projective
dimension\/~$\le\nobreak1$ in the abelian category\/ $\R\Contra$.
\end{prop}

\begin{proof}
 This is~\cite[Corollary~9.7]{Pflcc}.
\end{proof}

\Section{Acyclic Complexes of Flat Contramodules are Contraacyclic}
\label{acyclic-are-contraacyclic-secn}

 Let $\R$ be a complete, separated topological ring with a countable
base of neighborhoods of zero consisting of open right ideals.
 We will say that a complex of flat left $\R$\+contramodules $\fF^\bu$
is \emph{pure acyclic} if it is acyclic in the exact category
$\R\Contra_\flat$ \,\cite[Section~11]{Pflcc}.
 Equivalently, a complex of flat $\R$\+contramodules $\fF^\bu$ is pure
acyclic if and only if $\fF^\bu$ is acyclic in the abelian category
$\R\Contra$ \emph{and} the $\R$\+contramodules of cocycles of $\fF^\bu$
are flat.
 Let us denote the full subcategory of pure acyclic complexes of
flat left $\R$\+contramodules by $\Ac(\R\Contra_\flat)\subset
\Com(\R\Contra_\flat)\subset\Com(\R\Contra)$.

 The aim of this section is to prove the following theorem.

\begin{thm} \label{pure-acyclic-are-contraacyclic}
 Let\/ $\R$ be a complete, separated topological ring with
a \emph{countable} base of neighborhoods of zero consisting of
open \emph{two-sided} ideals.
 Then, for any complex of projective left\/ $\R$\+contramodules\/
$\fP^\bu$ and any pure acyclic complex of flat left\/
$\R$\+contramodules\/ $\fF^\bu$, any morphism of complexes\/
$\fP^\bu\rarrow\fF^\bu$ is homotopic to zero.
 In other words, any pure acyclic complex of flat\/ $\R$\+contramodules
is contraacyclic (as a complex in\/ $\R\Contra$, or equivalently,
in\/ $\R\Contra_\flat$).
\end{thm}

 We start with collecting some results from~\cite{PS4}
and~\cite{Pflcc} on which the proof of
Theorem~\ref{pure-acyclic-are-contraacyclic} is based.

\begin{prop} \label{complexes-of-projectives-filtered-by}
 Let $\kappa$~be a regular cardinal and\/ $\sB$ be an abelian category
with infinite coproducts and a $\kappa$\+presentable projective
generator~$P$ (these conditions imply that the category\/ $\sB$ is
locally $\kappa$\+presentable).
 Denote by\/ $\sS$ the set of all (representatives of isomorphism
classes of) bounded below complexes in\/ $\sB$ whose terms are
coproducts of less then~$\kappa$ copies of~$P$.
 Then any complex of projective objects in\/ $\sB$ is a direct summand
of a complex filtered by complexes from\/~$\sS$.
\end{prop}

\begin{proof}
 Notice, first of all, that any projective object in $\sB$ is a direct
summand of a coproduct of copies of~$P$; hence any complex of
projective objects in $\sB$ is a direct summand of a complex whose
components are coproducts of copies of~$P$.
 For the purposes of this paper, we are really interested in the case
$\kappa=\aleph_1$.
 When $\kappa=\aleph_0$, the category $\sB$ is the module category
$\sB\simeq R\Modl$, where $R$ is (the opposite ring to)
the endomorphism ring of the small projective generator $P\in\sB$.
 This case is covered by~\cite[Proposition~4.3]{Sto0}
(cf.~\cite[Construction~4.3]{Neem}).
 In the case of an uncountable regular cardinal~$\kappa$, the desired
assertion was established in~\cite[second proof of
Proposition~7.2]{PS4}.
\end{proof}

\begin{prop} \label{pure-acyclic-complexes-accessible}
 Let\/ $\R$ be a complete, separated topological ring with
a \emph{countable} base of neighborhoods of zero consisting of
open \emph{two-sided} ideals.
 Then the category\/ $\Ac(\R\Contra_\flat)$ of pure acyclic complexes
of flat left\/ $\R$\+contramodules is\/ $\aleph_1$\+accessible.
 A pure acyclic complex of flat\/ $\R$\+contramodules\/ $\fH^\bu$
is\/ $\aleph_1$\+presentable as an object of\/ $\Ac(\R\Contra_\flat)$
if and only if all the\/ $\R$\+contramodules\/ $\fH^n$, \,$n\in\boZ$,
are countably presentable.
\end{prop}

\begin{proof}
 This is~\cite[Corollary~11.4]{Pflcc}.
\end{proof}

\begin{proof}[Proof of Theorem~\ref{pure-acyclic-are-contraacyclic}]
 By Lemma~\ref{Ext-1-as-homotopy-Hom}, we have
$$
 \Hom_{\Hot(\R\Contra)}(\fP^\bu,\fF^\bu)\simeq
 \Ext^1_{\Com(\R\Contra)}(\fP^\bu,\fF^\bu[-1]).
$$
 Proposition~\ref{complexes-of-projectives-filtered-by} for
$\sB=\R\Contra$, \ $\kappa=\aleph_1$, and $P=\R=\R[\{*\}]\in
\R\Contra$ tells us that $\fP^\bu$ is a direct summand of a complex
filtered by complexes of free $\R$\+contramodules with (at most)
countable sets of generators.
 In view of Lemma~\ref{eklof-lemma}, the question reduces to the case
when $\fP^\bu$ is a complex of free $\R$\+contramodules with countable
sets of generators.

 Then, by Propositions~\ref{complexes-presentable-accessible}(a)
and~\ref{presentable-contramodules}, the complex $\fP^\bu$ is
a countably presentable object of the category $\sK=\Com(\R\Contra)$.
 On the other hand, by
Proposition~\ref{pure-acyclic-complexes-accessible}, the complex
$\fF^\bu$ is an $\aleph_1$\+directed colimit of pure acyclic complexes
of countably presentable flat $\R$\+contramodules.
 Notice that the full subcategory of pure acyclic complexes of flat
$\R$\+contramodules is closed under directed colimits in
$\Com(\R\Contra)$ by
Lemma~\ref{flat-contramodules-exactness-properties}(b); so
the directed colimits in this full subcategory agree with the ones
in $\Com(\R\Contra)$.
 Consequently, any morphism of complexes $\fP^\bu\rarrow\fF^\bu$
factorizes through some pure acyclic complex of countably presentable
flat $\R$\+contramodules~$\fH^\bu$.

 It remains to recall that the homological dimension of the exact
category of countably presentable flat $\R$\+contramodules does not
exceed~$1$ by
Proposition~\ref{projective-dimension-of-flat-countably-presented}.
 By Lemma~\ref{psemi-remark21}, it follows that the complex $\fH^\bu$
is absolutely acyclic in the exact category of countably presentable
flat $\R$\+contramodules, and consequently also in $\R\Contra_\flat$
and in $\R\Contra$.
 By Lemma~\ref{absolutely-acyclic-are-contraacyclic}(a), any morphism
of complexes $\fP^\bu\rarrow\fH^\bu$ is homotopic to zero.
\end{proof}

\begin{rem}
 Theorem~\ref{pure-acyclic-are-contraacyclic} is a contramodule
generalization of the similar assertion in Neeman's
paper~\cite[Theorem~8.6]{Neem}, which is the module version.
 Our proof of Theorem~\ref{pure-acyclic-are-contraacyclic} is
inspired by, but independent of the proof of the related assertion
in~\cite{Neem}.

 Alternatively, it is possible to deduce our
Theorem~\ref{pure-acyclic-are-contraacyclic} from Neeman's result.
 To this end, one argues as follows.
 Let $\R\supset\fI_1\supset\fI_2\supset\fI_3\supset\dotsb$ be
a countable base of neighborhoods of zero in $\R$ consisting of
open two-sided ideals $\fI_n$, \,$n\ge1$
(as in Lemma~\ref{flat-contramodules-described}
and Proposition~\ref{flat-contramodules-accessible}).
 Put $R_n=\R/\fI_n$.

 Then, by Lemma~\ref{flat-contramodules-described}, for any flat
left $\R$\+contramodules $\fP$ and $\fF$, there is a natural isomorphism
of abelian groups
$$
 \Hom^\R(\fP,\fF)\simeq\varprojlim\nolimits_{n\ge1}
 \Hom_{R_n}(\fP/\fI_n\tim\fP,\>\fF/\fI_n\tim\fF),
$$
where the transition maps in the projective system are induced by
the functors $R_n\ot_{R_{n+1}}{-}\,\:R_{n+1}\Modl\rarrow R_n\Modl$.
 When the $\R$\+contramodule $\fP$ is projective, the $R_n$\+modules
$\fP/\fI_n\tim\fP$ are projective, too, and it follows that any
$R_n$\+module map $\fP/\fI_n\tim\fP\rarrow\fF/\fI_n\tim\fF$ can be
lifted to an $R_{n+1}$\+module map $\fP/\fI_{n+1}\tim\fP\rarrow
\fF/\fI_{n+1}\tim\fF$.
 So the transition maps in the projective system are surjective.

 Now, for any complex of projective left $\R$\+contramodules $\fP^\bu$
and any complex of flat left $\R$\+contramodules $\fF^\bu$, we have
a natural isomorphism of complexes of abelian groups
$$
 \Hom^{\R,\bu}(\fP^\bu,\fF^\bu)\simeq\varprojlim\nolimits_{n\ge1}
 \Hom_{R_n}^\bu(\fP^\bu/\fI_n\tim\fP^\bu,\>\fF^\bu/\fI_n\tim\fF^\bu),
$$
where the transition maps are surjective in every degree.
 Whenever the complex of flat $\R$\+contramodules $\fF^\bu$ is pure
acyclic, so are the complexes of flat $R_n$\+modules $\fF/\fI_n\tim\fF$
(by Lemma~\ref{flat-contramodules-exactness-properties}(a)), so
the complexes
$\Hom_{R_n}^\bu(\fP^\bu/\fI_n\tim\fP^\bu,\>\fF^\bu/\fI_n\tim\fF^\bu)$
are acyclic by~\cite[Theorem~8.6]{Neem}.
 It remains to point out that the projective limit of a sequence
(indexed by the positive integers) of termwise surjective maps of
acyclic complexes of abelian groups is an acyclic complex again.
 Thus the complex $\Hom^{\R,\bu}(\fP^\bu,\fF^\bu)$ is acyclic whenever
the complex $\fF^\bu$ is pure acyclic.
\end{rem}

\Section{Contraacyclic Complexes of Flat Contramodules are Acyclic}
\label{contraacyclic-are-acyclic-secn}

 The aim of this section is to prove the following theorem, which
provides the converse implication to
Theorem~\ref{pure-acyclic-are-contraacyclic}.

\begin{thm} \label{contraacyclic-are-pure-acyclic}
 Let\/ $\R$ be a complete, separated topological ring with
a \emph{countable} base of neighborhoods of zero consisting of
open \emph{two-sided} ideals.
 Let\/ $\fF^\bu$ be a complex of flat left\/ $\R$\+contramodules.
 Assume that, for every complex of projective left\/
$\R$\+contramodules\/ $\fP^\bu$, all morphisms of complexes\/
$\fP^\bu\rarrow\fF^\bu$ are homotopic to zero.
 Then\/ $\fF^\bu$ is a pure acyclic complex.
 In other words, any contraacyclic complex of flat\/ $\R$\+contramodules
is pure acyclic.
\end{thm}

 Once again, we need to collect some results from~\cite{PR}
and~\cite{Pflcc} on which the proof of
Theorem~\ref{contraacyclic-are-pure-acyclic} is based, and also
prove some auxiliary lemmas.
 Let us start with the following definition.

 A left $\R$\+contramodule $\fC$ is said to be
\emph{cotorsion}~\cite[Definition~7.3]{PR} if one has
$\Ext^{\R,1}(\fF,\fC)=0$ for every flat left $\R$\+contramodule~$\fF$.
 We denote the full subcategory of cotorsion left
$\R$\+contramodules by $\R\Contra^\cot\subset\R\Contra$.
 The following lemma may help the reader to feel more comfortable.

\begin{lem} \label{cotorsion-contramodules-higher-ext}
 Let\/ $\R$ be a complete, separated topological ring with
a \emph{countable} base of neighborhoods of zero consisting of
open right ideals.
 Then one has\/ $\Ext^{\R,n}(\fF,\fC)=0$ for any flat left\/
$\R$\+contramodule\/ $\fF$, any cotorsion left\/ $\R$\+contramodule\/
$\fC$, and all $n\ge1$.
\end{lem}

\begin{proof}
 Compare Lemma~\ref{flat-contramodules-resolving} with
Lemma~\ref{right-orthogonal-to-self-resolving}.
\end{proof}

\begin{prop} \label{cotorsion-preenvelope}
 Let\/ $\R$ be a complete, separated topological ring with
a \emph{countable} base of neighborhoods of zero consisting of
open right ideals.
 Then every left\/ $\R$\+contramodule\/ $\fM$ can be included into
a short exact sequence\/ $0\rarrow\fM\rarrow\fC\rarrow\fG\rarrow0$
in\/ $\R\Contra$ with a cotorsion\/ $\R$\+contramodule\/ $\fC$
and a flat\/ $\R$\+con\-tra\-mod\-ule\/~$\fG$.
\end{prop}

\begin{proof}
 This is a part of~\cite[Corollary~7.8]{PR} (see the terminology
in~\cite[Section~3]{PR} or in Section~\ref{cotorsion-pairs-secn}
below).
\end{proof}

\begin{lem} \label{Ext-1-vanishes}
 Let\/ $\R$ be a complete, separated topological ring with
a \emph{countable} base of neighborhoods of zero consisting of
open right ideals.
 Let\/ $\fG$ be a left\/ $\R$\+contramodule that is a quotient
contramodule of a cotorsion\/ $\R$\+contramodule.
 Then one has\/ $\Ext^{\R,1}(\fH,\fG)=0$ for every flat left\/
$\R$\+contramodule\/ $\fH$ of projective dimension\/~$1$.
\end{lem}

\begin{proof}
 Let $0\rarrow\fM\rarrow\fC\rarrow\fG\rarrow0$ be a short exact
sequence in $\R\Contra$ with a cotorsion $\R$\+contramodule~$\fC$.
 Then we have a long exact sequence of abelian groups
$\dotsb\rarrow\Ext^{\R,1}(\fH,\fC)\rarrow\Ext^{\R,1}(\fH,\fG)
\rarrow\Ext^{\R,2}(\fH,\fM)\rarrow\dotsb$, implying the desired
$\Ext^{\R,1}$ vanishing.
\end{proof}

\begin{cor} \label{orthogonal-to-total-complexes}
 Let\/ $\R$ be a complete, separated topological ring with
a \emph{countable} base of neighborhoods of zero consisting of
open right ideals.
 Let\/ $\fG^\bu$ be a complex of left\/ $\R$\+contramodules whose
terms are quotient contramodules of cotorsion\/ $\R$\+contramodules.
 Let\/ $0\rarrow\fP^\bu\rarrow\fQ^\bu\rarrow\fH^\bu\rarrow0$ be
a short exact sequence of complexes of left\/ $\R$\+contramodules
such that\/ $\fH^\bu$ is a complex of flat\/ $\R$\+contramodules of
projective dimension at most\/~$1$.
 Denote by\/ $\fT^\bu=\Tot(\fP^\bu\to\fQ^\bu\to\fH^\bu)$ the related
total complex.
 Then every morphism of complexes\/ $\fT^\bu\rarrow\fG^\bu$ is
homotopic to zero.
\end{cor}

\begin{proof}
 In view of Lemma~\ref{Ext-1-vanishes}, we have a short exact
sequence of complexes of abelian groups $0\rarrow\Hom^{\R,\bu}(\fH^\bu,
\fG^\bu)\rarrow\Hom^{\R,\bu}(\fQ^\bu,\fG^\bu)\rarrow
\Hom^{\R,\bu}(\fP^\bu,\fG^\bu)\rarrow0$.
 So the dual argument to the proof of
Lemma~\ref{absolutely-acyclic-are-contraacyclic}(a) applies, proving
that the complex of abelian groups $\Hom^{\R,\bu}(\fT^\bu,\fG^\bu)$
is acyclic.
\end{proof}

\begin{cor} \label{into-contractible-of-flat-cotorsion}
 Let\/ $\R$ be a complete, separated topological ring with
a \emph{countable} base of neighborhoods of zero consisting of
open right ideals.
 Then every complex of flat left\/ $\R$\+contramodules\/ $\fF^\bu$
can be included into a short exact sequence\/ $0\rarrow\fF^\bu
\rarrow\fC^\bu\rarrow\fG^\bu\rarrow0$ in\/ $\Com(\R\Contra)$ with
a contractible complex of flat cotorsion\/ $\R$\+contramodules\/
$\fC^\bu$ and a complex of\/ flat\/ $\R$\+contramodules\/~$\fG^\bu$.
\end{cor}

\begin{proof}
 This is a corollary of Proposition~\ref{cotorsion-preenvelope}.
 For every cohomological degree $n\in\boZ$, choose a short exact
sequence $0\rarrow\fF^n\rarrow\overline\fC^n\rarrow
\overline\fG^n\rarrow0$ in $\R\Contra$ with a cotorsion
$\R$\+contramodule $\overline\fC^n$ and a flat
$\R$\+contramodule~$\overline\fG^n$.
 By Lemma~\ref{flat-contramodules-resolving}, the $\R$\+contramodules
$\overline\fC^n$ are flat as extensions of flat $\R$\+contramodules
$\fF^n$ and~$\overline\fG^n$.

 Let $\fC^\bu$ be the direct sum of contractible two-term complexes
$0\rarrow\overline\fC^n\overset\id\rarrow\overline\fC^n\rarrow0$
with the cohomological grading defined by the rule that
$\fC^n=\overline\fC^n\oplus\overline\fC^{n+1}$.
 Let $\fF^\bu\rarrow\fC^\bu$ be the morphism of complexes whose
components are the monomorphisms $\fF^n\rarrow\overline\fC^n$
and the compositions $\fF^n\rarrow\fF^{n+1}\rarrow\overline\fC^{n+1}$.
 Then $\fF^\bu\rarrow\fC^\bu$ is a (termwise) monomorphism of
complexes of $\R$\+contramodules.
 Let $\fG^\bu$ be the cokernel complex.
 Then there are short exact sequences of $\R$\+contramodules
$0\rarrow\overline\fC^{n+1}\rarrow\fG^n\rarrow\overline\fG^n\rarrow0$.
 Once again, Lemma~\ref{flat-contramodules-resolving} tells us that
the $\R$\+contramodules $\fG^n$ are flat as extensions of flat
$\R$\+contramodules.
\end{proof}

\begin{prop} \label{complexes-of-flats-accessible}
 Let\/ $\R$ be a complete, separated topological ring with
a \emph{countable} base of neighborhoods of zero consisting of
open \emph{two-sided} ideals.
 Then the category\/ $\Com(\R\Contra_\flat)$ of complexes of flat
left\/ $\R$\+contramodules is\/ $\aleph_1$\+accessible.
 A complex of flat\/ $\R$\+contramodules\/ $\fH^\bu$ is\/
$\aleph_1$\+presentable as an object of\/ $\Com(\R\Contra_\flat)$ if
and only if all the\/ $\R$\+contramodules\/ $\fH^n$, \,$n\in\boZ$,
are countably presentable.
\end{prop}

\begin{proof}
 This is~\cite[Proposition~10.2]{Pflcc}.
 The assertion follows from
Propositions~\ref{flat-contramodules-accessible}(a)
and~\ref{complexes-presentable-accessible}(b) (with $\kappa=\aleph_1$
and $\lambda=\aleph_0$).
\end{proof}

\begin{proof}[Proof of Theorem~\ref{contraacyclic-are-pure-acyclic}]
 Taking $\fP^\bu=\R[[\{*\}]]$ to be the one-term complex corresponding
to the free $\R$\+contramodule with one generator $\R[[\{*\}]]=\R$,
one can immediately see that any complex $\fF^\bu$ satisfying
the assumption of the theorem is acyclic in the abelian category
$\R\Contra$.
 Our task is to prove that the complex $\fF^\bu$ is pure acyclic, i.~e.,
acyclic in the exact category $\R\Contra_\flat$.

 Let\/ $\fF^\bu$ be a complex of flat left $\R$\+contramodules.
 Using Corollary~\ref{into-contractible-of-flat-cotorsion}, we can find
a short exact sequence of complexes of $\R$\+contramodules
$0\rarrow\fF^\bu\rarrow\fC^\bu\rarrow\fG^\bu\rarrow0$, where $\fC^\bu$
is a contractible complex of flat cotorsion $\R$\+contramodules
and $\fG^\bu$ is a complex of flat $\R$\+contramodules.
 Then the complex $\fC^\bu$, being a contractible complex of flat
contramodules, is both pure acyclic and contraacyclic.

 Since the class of all acyclic complexes in an exact category is
closed under kernels of termwise admissible epimorphisms and
cokernels of termwise admissible monomorphisms, the complex $\fF^\bu$
is pure acyclic if and only if the complex $\fG^\bu$ is pure acyclic.
 On the other hand, in view of
Lemma~\ref{absolutely-acyclic-are-contraacyclic}(a), the complex 
$\fF^\bu$ is contraacyclic if and only if the complex $\fG^\bu$
is contraacyclic.
 So we can consider $\fG^\bu$ instead of $\fF^\bu$ and assume that
$\fG^\bu$ is contraacyclic.
 Our aim is to prove that $\fG^\bu$ is pure acyclic.
 What we have achieved here is that we know additionally that all
the terms of the complex $\fG^\bu$ are quotient contramodules of
cotorsion contramodules.

 By Proposition~\ref{complexes-of-flats-accessible}, we have
an $\aleph_1$\+accessible category $\sK=\Com(\R\Contra_\flat)$.
 Furthermore, the $\aleph_1$\+presentable objects of $\sK$ are
precisely all the complexes of countably presentable flat
$\R$\+contramodules.
 The full subcategory of pure acyclic complexes of flat
$\R$\+contramodules is closed under directed colimits in $\sK$
by Lemma~\ref{flat-contramodules-exactness-properties}(b).
 Denote by $\sT\subset\sK$ the set of all (representatives of
isomorphism classes of) pure acyclic complexes of countably
presentable flat $\R$\+contramodules.
 Then all the objects from $\sT$ are $\aleph_1$\+presentable in~$\sK$.
 By Proposition~\ref{pure-acyclic-complexes-accessible}, the full
subcategory $\varinjlim_{(\aleph_1)}\sT\subset\sK$ coincides with
the full subcategory of pure acyclic complexes of flat
$\R$\+contramodules.
 The complex $\fG^\bu$ is an object of~$\sK$.
 Now Proposition~\ref{accessible-subcategory} tells us that, in order
to prove that $\fG^\bu$ belongs to $\varinjlim_{(\aleph_1)}\sT$,
it suffices to check that, for every complex of countably presentable
flat $\R$\+contramodules $\fH^\bu$, any morphism of complexes
$\fH^\bu\rarrow\fG^\bu$ factorizes through a complex from~$\sT$.

 It is easy to represent $\fH^\bu$ as a quotient complex of
a complex of (countably generated) projective
$\R$\+contramodules~$\fQ^\bu$.
 So we have a short exact sequence of complexes of $\R$\+contramodules
$0\rarrow\fP^\bu\rarrow\fQ^\bu\rarrow\fH^\bu\rarrow0$.
 By Proposition~\ref{projective-dimension-of-flat-countably-presented},
\,$\fP^\bu$ is also a complex of projective $\R$\+contramodules
(in fact, of countably generated projective $\R$\+contramodules,
by~\cite[Lemma~11.1]{Pflcc}).
 Put $\fT^\bu=\Tot(\fP^\bu\to\fQ^\bu\to\fH^\bu)$.
 Then the complex of abelian groups $\Hom^{\R,\bu}(\fT^\bu,\fG^\bu)$ is
acyclic by Corollary~\ref{orthogonal-to-total-complexes}.
 The complexes of abelian groups $\Hom^{\R,\bu}(\fP^\bu,\fG^\bu)$ and
$\Hom^{\R,\bu}(\fQ^\bu,\fG^\bu)$ are acyclic, since the complex
$\fG^\bu$ is contraacyclic in $\R\Contra$ by assumption.
 Thus the complex of abelian groups $\Hom^{\R,\bu}(\fH^\bu,\fG^\bu)$ is
acyclic, too.

 We have shown that every morphism of complexes $f\:\fH^\bu\rarrow
\fG^\bu$ is homotopic to zero.
 Thus the morphism~$f$ factorizes through the cone of the identity
endomorphism of the complex~$\fH^\bu$.
 The latter cone is a contractible (hence pure acyclic) complex of
countably presentable flat left $\R$\+contramodules.
 So it belongs to~$\sT$.
\end{proof}

\Section{Cotorsion Pairs} \label{cotorsion-pairs-secn}

 This section continues the discussion of category-theoretic background
that was started in Section~\ref{contraderived-secn}.
 Given an exact category $\sK$, a pair of classes of objects
$(\sF,\sC)$ in $\sK$ is said to be a \emph{cotorsion pair} if
$\sC=\sF^{\perp_1}$ and $\sF={}^{\perp_1}\sC$.

 For any class of objects $\sS\subset\sK$, the pair of classes
$\sC=\sS^{\perp_1}$ and $\sF={}^{\perp_1}\sC$ is a cotorsion pair
in~$\sK$.
 The cotorsion pair $(\sF,\sC)$ is said to be \emph{generated} by
the class of objects~$\sS$.

 Dually to the definition in Section~\ref{contraderived-secn}, a class
of objects $\sC\subset\sK$ is said to be \emph{cogenerating} if for
every object $K\in\sK$ there exists an object $C\in\sC$ together with
an admissible monomorphism $K\rarrow C$ in~$\sK$.

 Let $(\sF,\sC)$ be a cotorsion pair in $\sK$ such that the class $\sF$
is generating and the class $\sC$ is cogenerating in~$\sK$.
 Then the cotorsion pair $(\sF,\sC)$ is said to be \emph{hereditary}
if any one of the following equivalent conditions holds:
\begin{enumerate}
\item the class $\sF$ is closed under kernels of admissible epimorphisms
in~$\sK$;
\item the class $\sC$ is closed under cokernels of admissible
monomorphisms in~$\sK$;
\item $\Ext^2_\sK(F,C)=0$ for all $F\in\sF$ and $C\in\sC$;
\item $\Ext^n_\sK(F,C)=0$ for all $F\in\sF$, \ $C\in\sC$, and $n\ge1$.
\end{enumerate}
 The nontrivial implications are (1)~$\Longrightarrow$~(4) and
(2)~$\Longrightarrow$~(4); they hold by
Lemma~\ref{right-orthogonal-to-self-resolving} and its dual version.

 A cotorsion pair $(\sF,\sC)$ is $\sK$ is said to be \emph{complete} if
for every object $K\in\sK$ there exist (admissible) short exact
sequences
\begin{gather}
 0\lrarrow C'\lrarrow F\lrarrow K\lrarrow 0,
 \label{special-precover-sequence} \\
 0\lrarrow K\lrarrow C\lrarrow F'\lrarrow 0
 \label{special-preenvelope-sequence}
\end{gather}
in~$\sK$ with objects $F$, $F'\in\sF$ and $C$, $C'\in\sC$.
 The short exact sequence~\eqref{special-precover-sequence} is said to
be a \emph{special precover sequence}.
 The short exact sequence~\eqref{special-preenvelope-sequence} is said
to be a \emph{special preenvelope sequence}.

 The following result is known classically as the \emph{Eklof--Trlifaj
theorem}~\cite[Theorems~2 and~10]{ET}.
 Here, given a class of objects $\sF\subset\sK$, we denote by
$\sF^\oplus\subset\sK$ the class of all direct summands of
the objects from~$\sF$.
 The notation $\Fil(\sS)\subset\sK$ was introduced in
Section~\ref{contraderived-secn}.

\begin{thm} \label{eklof-trlifaj-theorem}
 Let\/ $\sB$ be a locally presentable abelian category and\/ $\sS\subset
\sB$ be a \emph{set} of objects.
 Consider the cotorsion pair $(\sF,\sC)$ in\/ $\sB$ generated
by\/~$\sS$.
 In this context: \par
\textup{(a)} if the class\/ $\sF$ is generating in\/ $\sB$ and
the class\/ $\sC$ is cogenerating in\/ $\sB$, then the cotorsion
pair $(\sF,\sC)$ is complete; \par
\textup{(b)} if the class\/ $\Fil(\sS)$ is generating in\/ $\sB$,
then\/ $\sF=\Fil(\sS)^\oplus$.
\end{thm}

\begin{proof}
 Part~(a) is~\cite[Corollary~3.6]{PR}.
 Part~(b) is~\cite[Theorem~4.8(a,d)]{PR}.
 See also~\cite[Theorems~3.3 and~3.4]{PS4}.
\end{proof}

 The following result is closely related to
Theorem~\ref{becker-contraderived-of-lpacepo}.

\begin{thm} \label{contraderived-cotorsion-pair-for-lpacepo}
 Let\/ $\sB$ be a locally presentable abelian category with enough
projective objects.
 Then the pair of classes of complexes of projective objects\/
$\sF=\Com(\sB_\proj)$ and contraacyclic complexes\/
$\sC=\Ac^\bctr(\sB)$ is a hereditary complete cotorsion pair $(\sF,\sC)$
in the abelian category of complexes\/ $\Com(\sB)$.
\end{thm}

\begin{proof}
 This is~\cite[Theorem~7.3]{PS4}.
 For a generalization, see~\cite[Theorem~6.16]{PS5}.
\end{proof}

 Let $(\sF,\sC)$ be a complete cotorsion pair in $\sK$ and
$\sE\subset\sK$ be a full subcategory closed under extensions.
 We will say that the cotorsion pair $(\sF,\sC)$ \emph{restricts} to
a complete cotorsion pair in $\sE$ if the pair of classes $\sE\cap\sF$
and $\sE\cap\sC$ is a complete cotorsion pair in (the inherited exact
structure on)~$\sE$.

\begin{lem} \label{cotorsion-pair-restricts}
 Let\/ $\sK$ be an exact category and\/ $\sE\subset\sK$ be a full
subcategory closed under extensions.
 Let $(\sF,\sC)$ be a complete cotorsion pair in\/~$\sK$.
 In this setting: \par
\textup{(a)} if\/ $\sE$ is closed under kernels of admissible
epimorphisms in\/ $\sK$ and\/ $\sF\subset\sE$, then the cotorsion pair
$(\sF,\sC)$ in\/ $\sK$ restricts to a complete cotorsion pair
$(\sF,\,\sE\cap\sC)$ in\/~$\sE$; \par
\textup{(b)} if\/ $\sE$ is closed under cokernels of admissible
monomorphisms in\/ $\sK$ and\/ $\sC\subset\sE$, then the cotorsion pair
$(\sF,\sC)$ in\/ $\sK$ restricts to a complete cotorsion pair
$(\sE\cap\sF,\,\sC)$ in\/~$\sE$. \par
 Furthermore, in both cases the restricted cotorsion pair
$(\sE\cap\sF,\,\sE\cap\sC)$ in\/ $\sE$ is hereditary whenever
the cotorsion pair $(\sF,\sC)$ in\/ $\sK$ is hereditary.
\end{lem}

\begin{proof}
 This is~\cite[Lemmas~1.5 and~1.6]{Pal}.
\end{proof}

 An exact category $\sE$ is said to have \emph{exact directed colimits}
if all directed colimits exist in $\sE$ and preserve (admissible)
short exact sequences.
 The following important proposition goes back
to~\cite[proofs of Lemma~5.2 and Proposition~5.3]{Sto},
\cite[Propositions~3.1 and~3.2]{Gil4},
and~\cite[proofs of Lemma~4.6 and Theorem~4.7]{BCE}.

\begin{prop} \label{transfinite-extensions-directed-colimits}
 Let\/ $\sE$ be an exact category with exact directed colimits and\/
$\sA\subset\sE$ be a full subcategory closed under extensions and
cokernels of admissible monomorphisms in\/~$\sE$.
 Then\/ $\sA$ is closed under transfinitely iterated extensions if and
only if\/ $\sA$ is closed under directed colimits in\/~$\sE$.
\end{prop}

\begin{proof}
 This is~\cite[Proposition~8.1]{PS6}.
\end{proof}

 In the following lemma, the class of objects $\sF={}^{\perp_1}\sC
\subset\sK$ is viewed as a full subcategory of $\sK$ endowed with
the exact structure inherited from~$\sK$.

\begin{lem} \label{hom-from-absolutely-acyclic-in-cotorsion-pair}
 Let\/ $\sK$ be an exact category and\/ $\sC\subset\sK$ be a full
subcategory.
 Put\/ $\sF={}^{\perp_1}\sC\subset\sK$.
 Let $C^\bu$ be a complex in\/~$\sC$.
 In this context: \par
\textup{(a)} For any (termwise admissible) short exact sequence\/
$0\rarrow F^\bu\rarrow G^\bu\rarrow H^\bu\rarrow0$ of complexes in\/
$\sF$, the short sequence of complexes of abelian groups\/ $0\rarrow
\Hom_\sK^\bu(H^\bu,C^\bu)\rarrow\Hom_\sK^\bu(G^\bu,C^\bu)\rarrow
\Hom_\sK^\bu(F^\bu,C^\bu)\rarrow0$ is exact. \par
\textup{(b)} For any absolutely acyclic complex $A^\bu$ in\/ $\sF$,
any morphism of complexes $A^\bu\rarrow C^\bu$ is homotopic to zero.
\end{lem}

\begin{proof}
 This is a generalization of the dual version of
Lemma~\ref{absolutely-acyclic-are-contraacyclic}(a).
 Part~(a) is obvious; and to prove part~(b), one can observe that
acyclicity of the complex of abelian groups $\Hom_\sK^\bu(A^\bu,C^\bu)$
follows from part~(a).
\end{proof}

 In the next lemma, both the classes $\sF$ and $\sC\subset\sK$ are
viewed as full subcategories endowed with the exact structures
inherited from~$\sK$.

\begin{lem} \label{acyclics-orthogonal-to-acyclics}
 Let\/ $\sK$ be an exact category and $(\sF,\sC)$ be a cotorsion pair
in\/~$\sK$.
 Let $F^\bu$ be an acyclic complex in the exact category\/ $\sF$ and
$C^\bu$ be an acyclic complex in the exact category\/~$\sC$.
 Then, viewing both $F^\bu$ and $C^\bu$ as complexes in\/ $\sK$,
any morphism of complexes $F^\bu\rarrow C^\bu$ is homotopic to zero.
\end{lem}

\begin{proof}
 This is~\cite[Lemma~3.9]{Gil} or~\cite[Lemma~B.1.8]{Pcosh}.
\end{proof}

\Section{Cotorsion Pair with Complexes of Cotorsion Contramodules}
\label{complexes-of-cotorsion-secn}

 The aim of this section is to prove a contramodule version of
the theorem of Bazzoni, Cort\'es-Izurdiaga, and
Estrada~\cite[Theorem~5.3]{BCE}.
 The argument in~\cite{BCE} is based on cotorsion periodicity together
with~\cite[Corollary~4.10]{Gil}.
 The proof of the latter uses considerations of deconstructibility in
the category of complexes.

 Such deconstructibility results have been established in full
generality for Grothendieck abelian categories~\cite[Section~4]{Sto0},
which does not cover our context.
 (Indeed, the directed colimits in $\R\Contra$ are not exact, while
$\R\Contra_\flat$ is not an abelian category.)
 Our proof in this section does not use that much knowledge about
deconstructibility, but is based on accessibility considerations
instead. {\hbadness=1550\par}

 We start with the following proposition going back
to~\cite[Lemma~1 and Proposition~2]{BBE}.
 The definition of a deconstructible class was given in
Section~\ref{contraderived-secn}.

\begin{prop} \label{flat-contramodules-deconstructible}
 Let\/ $\R$ be a complete, separated topological ring with
a \emph{countable} base of neighborhoods of zero consisting of open
right ideals.
 Then the class of all flat left\/ $\R$\+contramodules\/
$\R\Contra_\flat$ is deconstructible in\/ $\R\Contra$.
\end{prop} 

\begin{proof}
 This is~\cite[Corollary~7.6]{PR}; see also~\cite[Corollary~13.9]{PPT}
for a generalization.
\end{proof}

 Let $\R$ be a complete, separated topological ring with a countable
base of neighborhoods of zero consisting of open right ideals.
 The notation $\Ac(\R\Contra_\flat)$ was introduced in
Section~\ref{acyclic-are-contraacyclic-secn}, and the notation
$\R\Contra^\cot$ in Section~\ref{contraacyclic-are-acyclic-secn}.

 By~\cite[Corollary~7.8]{PR}, the pair of classes
$\sF=\R\Contra_\flat$ and $\sC=\R\Contra^\cot$ is a complete cotorsion
pair $(\sF,\sC)$ in the abelian category $\sB=\R\Contra$
(cf.\ Proposition~\ref{cotorsion-preenvelope}).
 By Lemma~\ref{flat-contramodules-resolving} or
Lemma~\ref{cotorsion-contramodules-higher-ext}, this cotorsion pair is
also hereditary.
 In the case of modules over an associative ring $R$, these results go
back to~\cite[Theorems~2 and~10]{ET} and~\cite[Proposition~2
and Theorem~3]{BBE}.

 The following
Theorem~\ref{pure-acyclic-flat-arbitrary-cotorsion-cotorsion-pair}
is our contramodule generalization of~\cite[Theorem~5.3]{BCE} (see
also~\cite[Theorem~3.3]{CET} or~\cite[Lemma~5.1.1(b)]{Pcosh} for
a quasi-coherent sheaf version).
 In order to prove it, we need more restrictive assumptions.

\begin{rem}
 Before formulating the theorem, let us say a few words about
relevant references.
 In the paper of Bazzoni, Cort\'es-Izurdiaga, and Estrada~\cite{BCE},
the assertion of~\cite[Theorem~5.3]{BCE} is deduced from the cotorsion
periodicity theorem~\cite[Theorem~5.1(2)]{BCE} and a result of
Gillespie~\cite[Corollary~4.10]{Gil}.
 In the context of quasi-coherent sheaves, the assertion
of~\cite[Lemma~5.1.1(b)]{Pcosh} is likewise deduced from
the cotorsion periodicity theorem for quasi-coherent sheaves.
 Both of these results for quasi-coherent sheaves were first stated
by Christensen, Estrada, and Thompson~\cite[Theorem~3.3]{CET}, but
the proof contained a gap.
 In fact, the argument in~\cite[first paragraph of the proof of
Lemma~3.2]{CET} is erroneous.
 Another proof of the cotorsion periodicity for quasi-coherent sheaves
appeared in the paper~\cite[Theorem~10.2 and Corollary~10.4]{PS6} by
the present author and \v St\!'ov\'\i\v cek, while the error
in~\cite{CET} was subsequently corrected by Estrada, Gillespie, and
Odaba\c si in~\cite[Theorem~6.3(2) and Remark~6.7]{EGilO}.
 See~\cite[Remark~10.3]{PS6} for a discussion.
\end{rem}

\begin{thm} \label{pure-acyclic-flat-arbitrary-cotorsion-cotorsion-pair}
 Let\/ $\R$ be a complete, separated topological ring with
a \emph{countable} base of neighborhoods of zero consisting of open
\emph{two-sided} ideals.
 Then the pair of classes of pure acyclic complexes of flat
contramodules\/ $\sF=\Ac(\R\Contra_\flat)$ and arbitrary complexes
of cotorsion contramodules\/ $\sC=\Com(\R\Contra^\cot)$ is a hereditary
complete cotorsion pair $(\sF,\sC)$ in the abelian category of complexes
of left\/ $\R$\+contramodules\/ $\sB=\Com(\R\Contra)$.
\end{thm}

\begin{proof}
 By Proposition~\ref{flat-contramodules-deconstructible}, there exists
a set of flat left $\R$\+contramodules $\sS_0\subset\R\Contra_\flat$
such that $\R\Contra_\flat=\Fil(\sS_0)$.
 Denote by $\sS\subset\Com(\R\Contra)$ the set of all contractible
two-term complexes of (flat) left $\R$\+contramodules
$$
 \dotsb\lrarrow0\lrarrow\fS\overset\id\lrarrow\fS\lrarrow0\lrarrow\dotsb
$$
(situated in arbitrary cohomological degrees) with $\fS\in\sS_0$.
 Then all contractible two-term complexes of flat left
$\R$\+contramodules, and in particular, all contractible two-term
complexes of projective left $\R$\+contramodules belong to $\Fil(\sS)$.
 By~\cite[Lemma~4.6(d)]{PR}, we have $\Fil(\Fil(\sS))=\Fil(\sS)$;
hence the class $\Fil(\sS)$ is closed under coproducts in~$\sB$.
 So all contractible complexes of projective left $\R$\+contramodules
belong to~$\Fil(\sS)$.
 It follows easily that the class $\Fil(\sS)$ is generating in~$\sB$.

 Let $(\sF',\sC')$ denote the cotorsion pair generated by $\sS$
in~$\sB$.
 By Lemmas~\ref{eklof-lemma} and~\ref{Ext-from-disk-complex} (for
$i=1$), we have $\sC'\subset\Com(\R\Contra^\cot)$.
 Conversely, by Lemma~\ref{Ext-1-as-homotopy-Hom},
$\Com(\R\Contra^\cot)\subset\sC'$.
 We have shown that $\sC'=\Com(\R\Contra^\cot)=\sC$. 

 The construction from the proof of
Corollary~\ref{into-contractible-of-flat-cotorsion}
(based on Proposition~\ref{cotorsion-preenvelope}) shows that
any complex of $\R$\+contramodules can be embedded as a subcomplex
into a complex of cotorsion $\R$\+contramodules.
 So the class $\sC'$ is cogenerating in $\Com(\R\Contra)$.
 Therefore, both parts of Theorem~\ref{eklof-trlifaj-theorem} are
applicable, and we can conclude that $(\sF',\sC')$ is a complete
cotorsion pair in $\sB$ with $\sF'=\Fil(\sS)$.

 Since the class of flat $\R$\+contramodules $\R\Contra_\flat$ is
closed under extensions and directed colimits in $\R\Contra$
(see Section~\ref{topological-rings-secn} and
Lemma~\ref{flat-contramodules-resolving}), it is also closed under
transfinitely iterated extensions.
 Therefore, we have $\sF'\subset\Com(\R\Contra_\flat)$.
 Since the class of acyclic complexes in any exact category is closed
under extensions, and the directed colimit functors are exact
in $\R\Contra_\flat$ by
Lemma~\ref{flat-contramodules-exactness-properties}(b), it follows
that, moreover, $\sF'\subset\Ac(\R\Contra_\flat)$.

 Let us show that the full subcategory $\sF'$ is closed under
directed colimits in $\Com(\R\Contra_\flat)$.
 Indeed, in order to prove that $\sF'$ is closed under cokernels of
admissible monomorphisms in $\Com(\R\Contra_\flat)$, consider a short
exact sequence $0\rarrow\fF^\bu\rarrow\fG^\bu\rarrow\fH^\bu\rarrow0$
in $\Com(\R\Contra_\flat)$.
 Then, for any complex $\fC^\bu\in\Com(\R\Contra^\cot)$, we have
a short exact sequence of complexes of abelian groups
$0\rarrow\Hom^{\R,\bu}(\fH^\bu,\fC^\bu)\rarrow\Hom^{\R,\bu}
(\fG^\bu,\fC^\bu)\rarrow\Hom^{\R,\bu}(\fF^\bu,\fC^\bu)\rarrow0$
by Lemma~\ref{hom-from-absolutely-acyclic-in-cotorsion-pair}(a).
 Therefore, if the complexes $\Hom^{\R,\bu}(\fG^\bu,\fC^\bu)$ and
$\Hom^{\R,\bu}(\fF^\bu,\fC^\bu)$ are acyclic, then so is
the complex $\Hom^{\R,\bu}(\fH^\bu,\fC^\bu)$.
 In view of Lemma~\ref{Ext-1-as-homotopy-Hom}, this means that one
has $\fH^\bu\in\sF'={}^{\perp_1}\sC$ whenever $\fF^\bu\in\sF'$
and $\fG^\bu\in\sF'$.
 On the other hand, the full subcategory $\sF'$ is closed under
transfinitely iterated extensions in $\sB=\Com(\R\Contra)$, hence also
in $\Com(\R\Contra_\flat)$, by Lemma~\ref{eklof-lemma}.
 Applying Proposition~\ref{transfinite-extensions-directed-colimits},
we come to the desired conclusion that $\sF'$ is closed under directed
colimits in $\Com(\R\Contra_\flat)$.

 By Proposition~\ref{pure-acyclic-complexes-accessible}, all
the complexes belonging to $\Ac(\R\Contra_\flat)$ are directed
colimits of pure acyclic complexes of countably presentable flat
$\R$\+contramodules.
 By Proposition~\ref{projective-dimension-of-flat-countably-presented}
and Lemma~\ref{psemi-remark21}, all pure acyclic complexes of
countably presentable flat $\R$\+contramodules are absolutely
acyclic as complexes in $\R\Contra_\flat$.
 By Lemmas~\ref{hom-from-absolutely-acyclic-in-cotorsion-pair}(b)
and~\ref{Ext-1-as-homotopy-Hom}, all absolutely acyclic complexes
in $\R\Contra_\flat$ belong to $\sF'={}^{\perp_1}\sC$.
 Thus all pure acyclic complexes of countably presentable flat
$\R$\+contramodules belong to~$\sF'$.
 Using the result of the previous paragraph, we can conclude that
$\Ac(\R\Contra_\flat)\subset\sF'$.
 We have shown that $\sF'=\Ac(\R\Contra_\flat)=\sF$.

 So we have constructed the desired complete cotorsion pair $(\sF,\sC)$
in~$\sB$.
 It remains to observe that the cotorsion pair $(\sF,\sC)$ is
hereditary, because the class $\sC$ is closed under cokernels of
admissible monomorphisms in~$\sB$.
 This follows from the fact that the class $\R\Contra^\cot$ is closed
under cokernels of admissible monomorphisms in $\R\Contra$ (see
the remarks preceding this theorem).
\end{proof}

\begin{cor} \label{acyclic-in-flats=orthogonal-to-all-of-cotorsion}
 Let\/ $\R$ be a complete, separated topological ring with
a \emph{countable} base of neighborhoods of zero consisting of open
\emph{two-sided} ideals.
 Let\/ $\fF^\bu$ be a complex of flat left\/ $\R$\+contramodules.
 Then\/ $\fF^\bu$ is a pure acyclic complex if and only if, for every
complex of cotorsion left\/ $\R$\+contramodules\/ $\fC^\bu$, all
morphisms of complexes\/ $\fF^\bu\rarrow\fC^\bu$ are homotopic to zero.
\end{cor}

\begin{proof}
 Follows from
Theorem~\ref{pure-acyclic-flat-arbitrary-cotorsion-cotorsion-pair}
in view of Lemma~\ref{Ext-1-as-homotopy-Hom}.
\end{proof}

\Section{Coderived Category of Flat Contramodules}
\label{coderived-of-flat-secn}

 We start with spelling out the definition dual to the one in
Section~\ref{contraderived-secn}.
 Let $\sE$ be an exact category with enough injective objects.
 A complex $A^\bu$ in $\sE$ is said to be \emph{coacyclic}
(\emph{in the sense of Becker}~\cite[Proposition~1.3.6(2)]{Bec})
if, for every complex of injective objects $J^\bu$ in $\sE$,
the complex of abelian groups $\Hom_\sE^\bu(A^\bu,J^\bu)$ is acyclic.
 The full subcategory of coacyclic complexes is denoted by
$\Ac^\bco(\sE)\subset\Hot(\sE)$ or $\Ac^\bco(\sE)\subset\Com(\sE)$.
 The triangulated Verdier quotient category
$$
 \sD^\bco(\sE)=\Hot(\sE)/\Ac^\bco(\sE)
$$
is called the (\emph{Becker}) \emph{coderived category} of~$\sE$.

 In this section, we are interested in the exact category of flat
$\R$\+contramodules $\sE=\R\Contra_\flat$.
 Consider the hereditary complete cotorsion pair $\sF=\R\Contra_\flat$
and $\sC=\R\Contra^\cot$ in the ambient abelian category
$\sB=\R\Contra$ (as per the discussion in
Section~\ref{complexes-of-cotorsion-secn}).
 Let us introduce the notation $\R\Contra_\flat^\cot$ for
the intersection of the two full subcategories
$$
 \R\Contra_\flat^\cot=\R\Contra_\flat\cap\R\Contra^\cot
 \subset\R\Contra.
$$
 By Lemma~\ref{cotorsion-pair-restricts}(a), the cotorsion pair
$(\sF,\sC)$ in $\sB$ restricts to a hereditary complete cotorsion pair
($\R\Contra_\flat$, $\R\Contra_\flat^\cot)$ in the exact subcategory
$\R\Contra_\flat=\sE\subset\sB$.
 In other words, this means that there are enough injective objects
in the exact category $\R\Contra_\flat$, and these are precisely all
the flat cotorsion $\R$\+contramodules.

\begin{cor} \label{coacyclic=acyclic-in-flats}
 Let\/ $\R$ be a complete, separated topological ring with
a \emph{countable} base of neighborhoods of zero consisting of open
\emph{two-sided} ideals.
 Let\/ $\fF^\bu$ be a complex of flat left\/ $\R$\+contramodules.
 Then\/ $\fF^\bu$ is a pure acyclic complex if and only if, for every
complex of flat cotorsion left\/ $\R$\+contramodules\/ $\fG^\bu$, all
morphisms of complexes\/ $\fF^\bu\rarrow\fG^\bu$ are homotopic to zero.
 In other words, a complex is coacyclic in the exact category\/
$\R\Contra_\flat$ if and only if it is acyclic in\/ $\R\Contra_\flat$.
\end{cor}

\begin{proof}
 Consider the abelian category $\sB=\Com(\R\Contra)$ and the hereditary
complete cotorsion pair $\sF=\Ac(\R\Contra_\flat)$ and
$\sC=\Com(\R\Contra^\cot)$ in $\sB$, as per
Theorem~\ref{pure-acyclic-flat-arbitrary-cotorsion-cotorsion-pair}.
 By Lemma~\ref{cotorsion-pair-restricts}(a), the cotorsion pair
$(\sF,\sC)$ in $\sB$ restricts to a hereditary complete cotorsion pair
$(\sF,\,\sE\cap\sC)$ in the exact subcategory $\Com(\R\Contra_\flat)
=\sE\subset\sB$.
 Now $\sE\cap\sC=\Com(\R\Contra_\flat^\cot)$ is the full subcategory of
all complexes of flat cotorsion $\R$\+contramodules.
 So we have $\sF={}^{\perp_1}(\sE\cap\sC)$ in the exact category~$\sE$.
 In view of Lemma~\ref{Ext-1-as-homotopy-Hom}, the assertion of
the corollary follows.
\end{proof}

\begin{cor} \label{becker-coderived-of-flats}
 Let\/ $\R$ be a complete, separated topological ring with
a \emph{countable} base of neighborhoods of zero consisting of open
\emph{two-sided} ideals.
 Then the inclusion of additive/exact categories\/
$\R\Contra_\flat^\cot\rarrow\R\Contra_\flat$ induces a triangulated
equivalence
$$
 \Hot(\R\Contra_\flat^\cot)\simeq
 \sD^\bco(\R\Contra_\flat)=\sD(\R\Contra_\flat).
$$
\end{cor}

\begin{proof}
 Corollary~\ref{coacyclic=acyclic-in-flats} tells us that
$\sD^\bco(\R\Contra_\flat)=\sD(\R\Contra_\flat)$.
 It follows immediately from the definitions that the functor
$\Hot(\R\Contra_\flat^\cot)\rarrow\sD^\bco(\R\Contra_\flat)$ is fully
faithful.
 In other to prove that it is a triangulated equivalence, it remains
to construct for every complex of flat $\R$\+contramodules $\fF^\bu$
a complex of flat cotorsion $\R$\+contramodules $\fG^\bu$ together
with a morphism of complexes $\fF^\bu\rarrow\fG^\bu$ with
a pure acyclic cone.

 For this purpose, we use
Theorem~\ref{pure-acyclic-flat-arbitrary-cotorsion-cotorsion-pair}.
 As explained in the proof of
Corollary~\ref{coacyclic=acyclic-in-flats}, the theorem implies
that the pair of classes $\Ac(\R\Contra_\flat)$ and
$\Com(\R\Contra_\flat^\cot)$ is a hereditary complete cotorsion pair
in $\Com(\R\Contra_\flat)$.
 Consider a special preenvelope
sequence~\eqref{special-preenvelope-sequence} from
Section~\ref{cotorsion-pairs-secn} for the object
$\fF^\bu\in\Com(\R\Contra_\flat)$ with respect to this complete
cotorsion pair.
 We get a short exact sequence of complexes of flat $\R$\+contramodules
$0\rarrow\fF^\bu\rarrow\fG^\bu\rarrow\fH^\bu\rarrow0$ with
$\fG^\bu\in\Com(\R\Contra_\flat^\cot)$ and
$\fH^\bu\in\Ac(\R\Contra_\flat)$.
 Since the total complex $\Tot(\fF^\bu\to\fG^\bu\to\fH^\bu)$ is pure
acyclic, it follows that the cone of the morphism of complexes
$\fF^\bu\rarrow\fG^\bu$ also belongs to $\Ac(\R\Contra_\flat)$.
\end{proof}

\Section{Contraderived Category of Cotorsion Contramodules}
\label{contraderived-of-cotorsion-secn}

 The aim of this section is to prove a contramodule generalization
of~\cite[Corollary~7.21]{Pphil}.
 The argument in~\cite[Section~7.7]{Pphil} for the existence of
the relevant cotorsion pair uses considerations of deconstructibility
in the category of complexes.
 As in Section~\ref{complexes-of-cotorsion-secn} above, our proof in
this section partly substitutes information about deconstructibility
by an accessibility argument.

 Let $\R$ be a complete, separated topological ring with a countable
base of neighborhoods of zero consisting of open right ideals.
 We are interested in the exact category of
cotorsion $\R$\+contramodules $\sE=\R\Contra^\cot$ (with the exact
structure inherited from the abelian exact structure on $\R\Contra$).
 Consider the hereditary complete cotorsion pair $\sF=\R\Contra_\flat$
and $\sC=\R\Contra^\cot$ in the abelian category $\sB=\R\Contra$.
 By Lemma~\ref{cotorsion-pair-restricts}(b), the cotorsion pair
$(\sF,\sC)$ in $\sB$ restricts to a hereditary complete cotorsion
pair ($\R\Contra_\flat^\cot$, $\R\Contra^\cot$) in the exact category
$\R\Contra^\cot=\sE\subset\sB$.
 In other words, this means that there are enough projective objects
in the exact category $\R\Contra^\cot$, and these are precisely all
the flat cotorsion $\R$\+contramodules.

 In the following theorem, which is our contramodule generalization
of a construction of complete cotorsion pair from~\cite[proof of
Theorem~7.19]{Pphil}, we need more restrictive assumptions.
 (See also~\cite[Lemma~5.1.9(b)]{Pcosh} for a quasi-coherent sheaf
version, and~\cite[Proposition~3.2, Theorem~5.5, and Section~5.3]{Gil2}
or~\cite[Lemma~4.9]{Gil3} for an abstract  approach covering
Grothendieck categories.)

\begin{thm} \label{arbitrary-flat-contraacyclic-cotors-cotorsion-pair}
 Let\/ $\R$ be a complete, separated topological ring with
a \emph{countable} base of neighborhoods of zero consisting of open
\emph{two-sided} ideals.
 Then the pair of classes of arbitrary complexes of flat contramodules\/
$\sF=\Com(\R\Contra_\flat)$ and complexes of cotorsion contramodules
\emph{contraacyclic as complexes in} $\R\Contra$, i.~e.,
$\sC=\Com(\R\Contra^\cot)\cap\Ac^\bctr(\R\Contra)$, is a hereditary
complete cotorsion pair $(\sF,\sC)$ in the abelian category of
complexes of left\/ $\R$\+contramodules\/ $\sB=\Com(\R\Contra)$.
\end{thm}

\begin{proof}
 The argument bears some similarity to the proof of
Theorem~\ref{pure-acyclic-flat-arbitrary-cotorsion-cotorsion-pair}.
 By Proposition~\ref{flat-contramodules-deconstructible}, there exists
a set of flat $\R$\+contramodules $\sS_0\subset\R\Contra_\flat$
such that $\R\Contra_\flat=\Fil(\sS_0)$.
 Denote by $\sS_1\subset\Com(\R\Contra)$ the set of all contractible
two-term complexes of (flat) $\R$\+contramodules
$$
 \dotsb\lrarrow0\lrarrow\fS\overset\id\lrarrow\fS\lrarrow0\lrarrow\dotsb
$$
(situated in arbitrary cohomological degrees) with $\fS\in\sS_0$.
 Let $\sS_2\subset\Com(\R\Contra)$ be the set of (representatives of
the isomorphism classes of) all complexes of countably presentable
flat $\R$\+contramodules.
 By Proposition~\ref{complexes-of-flats-accessible}, all the complexes
of flat $\R$\+contramodules are directed colimits of complexes
from~$\sS_2$.
 Put $\sS=\sS_1\cup\sS_2$.
 As explained in the proof of
Theorem~\ref{pure-acyclic-flat-arbitrary-cotorsion-cotorsion-pair},
the class $\Fil(\sS)$ (and even $\Fil(\sS_1)$) is generating in~$\sB$.

 Let $(\sF',\sC')$ denote the cotorsion pair generated by $\sS$
in~$\sB$.
 By Lemmas~\ref{eklof-lemma} and~\ref{Ext-from-disk-complex} (for
$i=1$), we have $\sC'\subset\Com(\R\Contra^\cot)$.
 Conversely, by Lemma~\ref{Ext-1-as-homotopy-Hom}, all contractible
complexes of cotorsion left $\R$\+contramodules belong to~$\sC'$.

 The construction from the proof of
Corollary~\ref{into-contractible-of-flat-cotorsion} (based on
Proposition~\ref{cotorsion-preenvelope}) shows that any complex of
$\R$\+contramodules can be embedded as a subcomplex into
a contractible complex of cotorsion $\R$\+contramodules.
 So the class $\sC'$ is cogenerating in $\Com(\R\Contra)$.
 Therefore, both parts of Theorem~\ref{eklof-trlifaj-theorem} are
applicable, and we have shown that $(\sF',\sC')$ is a complete
cotorsion pair in $\sB$ with $\sF'=\Fil(\sS)$.
 Similarly to the proof of
Theorem~\ref{pure-acyclic-flat-arbitrary-cotorsion-cotorsion-pair},
it follows that $\sF'\subset\Com(\R\Contra_\flat)$.

 Continuing to argue similarly to the proof of
Theorem~\ref{pure-acyclic-flat-arbitrary-cotorsion-cotorsion-pair},
one shows that the full subcategory $\sF'$ is closed under directed
colimits in $\Com(\R\Contra_\flat)$.
 Therefore, $\sF'=\Com(\R\Contra_\flat)=\sF$ (since $\sS_2\subset\sF'$).
 In particular, all the complexes of projective left $\R$\+contramodules
belong to~$\sF'$.
 By the definition and in view of Lemma~\ref{Ext-1-as-homotopy-Hom},
it follows that $\sC'\subset\Ac^\bctr(\R\Contra)$.

 Conversely, let $\fC^\bu$ be a complex of cotorsion $\R$\+contramodules
that is contraacyclic as a complex in $\R\Contra$.
 Let $\fF^\bu$ be a complex of flat left $\R$\+contramodules.
 Let us show that $\Ext^1_\sB(\fF^\bu,\fC^\bu)=0$.
 To this end, we will need to recall yet another complete
cotorsion pair.

 According to Theorem~\ref{contraderived-cotorsion-pair-for-lpacepo},
the pair of classes ($\Com(\R\Contra_\proj)$, $\Ac^\bctr(\R\Contra)$)
is a hereditary complete cotorsion pair in $\Com(\R\Contra)$.
 By Lemma~\ref{cotorsion-pair-restricts}(a), this cotorsion pair
restricts to a hereditary complete cotorsion pair in the exact
subcategory $\Com(\R\Contra_\flat)\subset\Com(\R\Contra)$ formed by
the pair of classes $\Com(\R\Contra_\proj)$ and
$\Com(\R\Contra_\flat)\cap\Ac^\bctr(\R\Contra)$.
 According to Theorems~\ref{pure-acyclic-are-contraacyclic}
and~\ref{contraacyclic-are-pure-acyclic}, we have
$\Com(\R\Contra_\flat)\cap\Ac^\bctr(\R\Contra)=\Ac(\R\Contra_\flat)$.
 We have shown that the pair of classes
($\Com(\R\Contra_\proj)$, $\Ac(\R\Contra_\flat)$) is a hereditary
complete cotorsion pair in $\Com(\R\Contra_\flat)$. {\hfuzz=4.2pt\par}

 For our purposes, we need either a special precover
sequence~\eqref{special-precover-sequence}, or a special preenvelope
sequence~\eqref{special-preenvelope-sequence} for the complex $\fF^\bu$
with respect to this cotorsion pair in $\Com(\R\Contra_\flat)$.
 Both are equally suitable.
 For example, let $0\rarrow\fH^\bu\rarrow\fP^\bu\rarrow\fF^\bu\rarrow0$
be a special precover sequence; so $\fP^\bu\in\Com(\R\Contra_\proj)$
and $\fH^\bu\in\Ac(\R\Contra_\flat)$.
 Since $\fF^\bu\in\Com(\R\Contra_\flat)$ and $\fC^\bu\in
\Com(\R\Contra^\cot)$, we have a short exact sequence of complexes of
abelian groups $0\rarrow\Hom^{\R,\bu}(\fF^\bu,\fC^\bu)\rarrow
\Hom^{\R,\bu}(\fP^\bu,\fC^\bu)\rarrow\Hom^{\R,\bu}(\fH^\bu,\fC^\bu)
\rarrow0$.

 Now the complex $\Hom^{\R,\bu}(\fP^\bu,\fC^\bu)$ is acyclic since
$\fC^\bu\in\Ac^\bctr(\R\Contra)$.
 Furthermore, the complex $\fC^\bu$ is acyclic in $\R\Contra$ by
Lemma~\ref{contraacyclic-in-abelian-are-acyclic}, hence it is also
acyclic in $\R\Contra^\cot$ by the cotorsion periodicity
theorem for contramodules~\cite[Corollary~12.4]{Pflcc}.
 As the complex $\fH^\bu$ is acyclic in $\R\Contra_\flat$, it follows
by virtue of Lemma~\ref{acyclics-orthogonal-to-acyclics} that
the complex $\Hom^{\R,\bu}(\fH^\bu,\fC^\bu)$ is acyclic.
 Therefore, the complex $\Hom^{\R,\bu}(\fF^\bu,\fC^\bu)$ is
acyclic, too.
 Applying Lemma~\ref{Ext-1-as-homotopy-Hom}, we conclude that
$\Ext^1_\sB(\fF^\bu,\fC^\bu)=0$.
 Thus $\sC'=\Com(\R\Contra^\cot)\cap\Ac^\bctr(\R\Contra)=\sC$.

 We have obtained the desired complete cotorsion pair $(\sF,\sC)$
in $\Com(\R\Contra)$, and it only remains to show that it is hereditary.
 Indeed, the class $\sF$ is closed under kernels of admissible
epimorphisms in $\sB$ by Lemma~\ref{flat-contramodules-resolving}.
\end{proof}

\begin{cor} \label{cotorsion-ctracycl-in-all=ctracycl-in-cotorsion}
 Let\/ $\R$ be a complete, separated topological ring with
a \emph{countable} base of neighborhoods of zero consisting of open
\emph{two-sided} ideals.
 Let\/ $\fC^\bu$ be a complex of cotorsion left\/ $\R$\+contramodules.
 Then the following two conditions are equivalent:
\begin{enumerate}
\item for every complex of projective left\/ $\R$\+contramodules\/
$\fP^\bu$, any morphism of complexes\/ $\fP^\bu\rarrow\fC^\bu$ is
homotopic to zero;
\item for every complex of flat cotorsion left\/ $\R$\+contramodules\/
$\fG^\bu$, any morphism of complexes\/ $\fG^\bu\rarrow\fC^\bu$ is
homotopic to zero.
\end{enumerate}
 In other words, $\fC^\bu$ is contraacyclic as a complex in\/
$\R\Contra$ if and only if\/ $\fC^\bu$ is contraacyclic as a complex
in\/ $\R\Contra^\cot$.
\end{cor}

\begin{proof}
 Consider the abelian category $\sB=\Com(\R\Contra)$ and the hereditary
complete cotorsion pair $\sF=\Com(\R\Contra_\flat)$ and
$\sC=\Com(\R\Contra^\cot)\cap\Ac^\bctr(\R\Contra)$ in $\sB$, as per
Theorem~\ref{arbitrary-flat-contraacyclic-cotors-cotorsion-pair}.
 By Lemma~\ref{cotorsion-pair-restricts}(b), the cotorsion pair
$(\sF,\sC)$ in $\sB$ restricts to a hereditary complete cotorsion pair
$(\sE\cap\sF,\,\sC)$ in the exact subcategory $\Com(\R\Contra^\cot)
=\sE\subset\sB$.
 Now $\sE\cap\sF=\Com(\R\Contra_\flat^\cot)$ is the full subcategory of
all complexes of flat cotorsion $\R$\+contramodules.
 So we have $\sC=(\sE\cap\sF)^{\perp_1}$ in the exact category~$\sE$.
 In view of Lemma~\ref{Ext-1-as-homotopy-Hom}, the assertion of
the corollary follows.
\end{proof}

\begin{cor} \label{becker-contraderived-of-cotorsion}
 Let\/ $\R$ be a complete, separated topological ring with
a \emph{countable} base of neighborhoods of zero consisting of open
\emph{two-sided} ideals.
 Then the inclusion of additive/exact categories\/ $\R\Contra_\flat^\cot
\rarrow\R\Contra^\cot$ induces a triangulated equivalence
$$
 \Hot(\R\Contra_\flat^\cot)\simeq\sD^\bctr(\R\Contra^\cot).
$$
\end{cor}

\begin{proof}
 Corollary~\ref{cotorsion-ctracycl-in-all=ctracycl-in-cotorsion} tells
us that $\Ac^\bctr(\R\Contra^\cot)=\Com(\R\Contra^\cot)\cap
\Ac^\bctr(\R\Contra)$.
 It follows immediately from the definitions that the functor
$\Hot(\R\Contra_\flat^\cot)\rarrow\sD^\bctr(\R\Contra^\cot)$ is fully
faithful.
 In order to prove that it is a triangulated equivalence, it remains
to construct for every complex of cotorsion $\R$\+contramodules
$\fC^\bu$ a complex of flat cotorsion $\R$\+contramodules $\fG^\bu$
together with a morphism of complexes $\fG^\bu\rarrow\fC^\bu$ with
a contraacyclic cone.

 For this purpose, we use
Theorem~\ref{arbitrary-flat-contraacyclic-cotors-cotorsion-pair}.
 As explained in the proof of
Corollary~\ref{cotorsion-ctracycl-in-all=ctracycl-in-cotorsion},
the theorem implies that the pair of classes
$\Com(\R\Contra_\flat^\cot)$ and $\Ac^\bctr(\R\Contra^\cot)$ is
a hereditary complete cotorsion pair in $\Com(\R\Contra^\cot)$.
 Consider a special precover sequence~\eqref{special-precover-sequence}
for the object $\fC^\bu\in\Com(\R\Contra^\cot)$ with respect to this
complete cotorsion pair.
 We get a short exact sequence of complexes of cotorsion
$\R$\+contramodules $0\rarrow\fB^\bu\rarrow\fG^\bu\rarrow\fC^\bu
\rarrow0$ with $\fG^\bu\in\Com(\R\Contra_\flat^\cot)$ and
$\fB^\bu\in\Ac^\bctr(\R\Contra^\cot)$.
 Since the total complex $\Tot(\fB^\bu\to\fG^\bu\to\fC^\bu)$ is
contraacyclic in $\R\Contra^\cot$ by
Lemma~\ref{absolutely-acyclic-are-contraacyclic}(a), it follows that
the cone of the morphism of complexes $\fG^\bu\rarrow\fC^\bu$ also
belongs to $\Ac^\bctr(\R\Contra^\cot)$.
\end{proof}

\Section{Five Constructions of the Contraderived Category}
\label{five-constructions-secn}

 Now we can formulate the theorem promised in the introduction,
summarizing our results generalizing Neeman's
theorem~\cite[Theorem~8.6]{Neem} and the related implications of
the theorem of Bazzoni, Cort\'es-Izurdiaga, and
Estrada~\cite[Theorem~5.3]{BCE}.

\begin{thm} \label{main-theorem}
 Let\/ $\R$ be a complete, separated topological ring with
a \emph{countable} base of neighborhoods of zero consisting of
open \emph{two-sided} ideals.
 Let\/ $\fF^\bu$ be a complex of flat left\/ $\R$\+contramodules.
 Then the following six conditions are equivalent:
\begin{enumerate}
\renewcommand{\theenumi}{\roman{enumi}${}^{\mathrm{c}}$}
\item for every complex of projective left\/ $\R$\+contramodules\/
$\fP^\bu$, any morphism of complexes of\/ $\R$\+contramodules\/
$\fP^\bu\rarrow\fF^\bu$ is homotopic to zero;
\item $\fF^\bu$ is an acyclic complex of\/ $\R$\+contramodules with
flat\/ $\R$\+contramodules of cocycles;
\item $\fF^\bu$ can be obtained from contractible complexes of flat\/
$\R$\+contramodules using extensions and directed colimits;
\item $\fF^\bu$ is an\/ $\aleph_1$\+directed colimit of total complexes
of short exact sequences of complexes of countably presentable flat\/
$\R$\+contramodules;
\item for every complex of cotorsion left\/ $\R$\+contramodules\/
$\fC^\bu$, any morphism of complexes of\/ $\R$\+contramodules\/ $\fF^\bu
\rarrow\fC^\bu$ is homotopic to zero;
\item for every complex of flat cotorsion left\/ $\R$\+contramodules\/
$\fG^\bu$, any morphism of complexes of\/ $\R$\+contramodules\/ $\fF^\bu
\rarrow\fG^\bu$ is homotopic to zero.
\end{enumerate}
\end{thm}

\begin{proof}
 The implication
(ii$^{\mathrm{c}}$)~$\Longrightarrow$~(i$^{\mathrm{c}}$) is
Theorem~\ref{pure-acyclic-are-contraacyclic}.

 The implication
(i$^{\mathrm{c}}$)~$\Longrightarrow$~(ii$^{\mathrm{c}}$) is
Theorem~\ref{contraacyclic-are-pure-acyclic}.

 The equivalences (ii$^{\mathrm{c}}$) $\Longleftrightarrow$
(iii$^{\mathrm{c}}$) $\Longleftrightarrow$ (iv$^{\mathrm{c}}$)
are provided by~\cite[Theorem~13.2]{Pflcc}.

 The equivalence (ii$^{\mathrm{c}}$) $\Longleftrightarrow$
(v$^{\mathrm{c}}$) is
Corollary~\ref{acyclic-in-flats=orthogonal-to-all-of-cotorsion}.

 The equivalence (ii$^{\mathrm{c}}$) $\Longleftrightarrow$
(vi$^{\mathrm{c}}$) is
Corollary~\ref{coacyclic=acyclic-in-flats}.
\end{proof}

 Finally, we deduce a corollary about an equivalence of the ``flat'', 
``projective'', ``flat cotorsion'', ``Verdier quotient by
contraacyclic'', and ``Verdier quotient of cotorsion by contraacyclic''
constructions of the contraderived category $\sD^\bctr(\R\Contra)$.
 It should be compared with~\cite[Theorems~4.12, 4.13, and~5.7]{PS7}.

\begin{cor}
 Let\/ $\R$ be a complete, separated topological ring with
a \emph{countable} base of neighborhoods of zero consisting of
open \emph{two-sided} ideals.
 Then the inclusions of additive/exact/abelian categories\/
$\R\Contra_\proj\rarrow\R\Contra_\flat\rarrow\R\Contra$ induce
equivalences of triangulated categories
$$
 \Hot(\R\Contra_\proj)\simeq\sD(\R\Contra_\flat)=
 \sD^\bctr(\R\Contra_\flat)\simeq\sD^\bctr(\R\Contra).
$$
 Furthermore, the inclusions of additive/exact/abelian categories\/
$\R\Contra_\flat^\cot\rarrow\R\Contra_\flat$ and\/
$\R\Contra_\flat^\cot\rarrow\R\Contra^\cot\rarrow\R\Contra$
induce equivalences of triangulated categories
$$
 \Hot(\R\Contra_\flat^\cot)\simeq\sD^\bco(\R\Contra_\flat)=
 \sD(\R\Contra_\flat)
$$
and
$$
 \Hot(\R\Contra_\flat^\cot)\simeq\sD^\bctr(\R\Contra^\cot)
 \simeq\sD^\bctr(\R\Contra).
$$
\end{cor}

\begin{proof}
 The functor $\Hot(\R\Contra_\proj)\rarrow\sD^\bctr(\R\Contra)$ is
a triangulated equivalence by
Theorem~\ref{becker-contraderived-of-lpacepo}.
 In fact, Theorem~\ref{becker-contraderived-of-lpacepo}
provides more information than it explicitly says.
 For every complex of left $\R$\+contramodules $\fM^\bu$, there exists
a complex of projective left $\R$\+contramodules $\fP^\bu$ together
with a morphism of complexes $\fP^\bu\rarrow\fM^\bu$ with
a contraacyclic cone.

 The assertion that the functor $\sD(\R\Contra_\proj)\rarrow
\sD(\R\Contra_\flat)$ is a triangulated equivalence is a corollary
of Theorem~\ref{main-theorem}\,%
(i$^{\mathrm{c}}$)\,$\Leftrightarrow$\,(ii$^{\mathrm{c}}$).
 We apply Lemma~\ref{pkoszul-lemma16}.
 Consider a complex of flat left\/ $\R$\+contramodules~$\fF^\bu$.
 Then, as a particular case of the previous paragraph, there exists
a complex of projective left $\R$\+contramodules $\fP^\bu$ together
with a morphism of complexes $\fP^\bu\rarrow\fF^\bu$ with
a contraacyclic cone.
 The cone, being a contraacyclic complex of flat $\R$\+contramodules,
is consequently an acyclic complex in the exact category
$\R\Contra_\flat$.
 So $\fP^\bu\rarrow\fF^\bu$ is an isomorphism in $\sD(\R\Contra_\flat)$.
 Furthermore, every complex of projective $\R$\+contramodules that is
acyclic in $\R\Contra_\flat$ is contractible, again
by Theorem~\ref{main-theorem}\,%
(i$^{\mathrm{c}}$)\,$\Leftrightarrow$\,(ii$^{\mathrm{c}}$) or
by~\cite[Proposition~12.1]{Pflcc}.
 The desired conclusion follows.

 One has $\sD^\bctr(\R\Contra_\flat)=\sD(\R\Contra_\flat)=
\sD^\bco(\R\Contra_\flat)$ by
Theorem~\ref{main-theorem}\,%
(i$^{\mathrm{c}}$)\,$\Leftrightarrow$\,(ii$^{\mathrm{c}}$)%
\,$\Leftrightarrow$\,(vi$^{\mathrm{c}}$).
 The triangulated equivalence
$\Hot(\R\Contra_\flat^\cot)\simeq\sD(\R\Contra_\flat)$ is the result
of Corollary~\ref{becker-coderived-of-flats}.
 The triangulated equivalence
$\Hot(\R\Contra_\flat^\cot)\simeq\sD^\bctr(\R\Contra^\cot)$ is
the result of Corollary~\ref{becker-contraderived-of-cotorsion}.
{\hbadness=1050\par}

 Finally, the triangulated functor
$\sD^\bctr(\R\Contra^\cot)\rarrow\sD^\bctr(\R\Contra)$
induced by the inclusion of exact/abelian categories
$\R\Contra^\cot\rarrow\R\Contra$ is well-defined by
Corollary~\ref{cotorsion-ctracycl-in-all=ctracycl-in-cotorsion}.
 This functor is a triangulated equivalence, since both the composition
$\Hot(\R\Contra_\flat^\cot)\rarrow\sD^\bctr(\R\Contra_\flat)\rarrow
\sD^\bctr(\R\Contra)$ and the functor $\Hot(\R\Contra_\flat^\cot)
\rarrow\sD^\bctr(\R\Contra^\cot)$ are triangulated equivalences.
\end{proof}

\begin{qst}
 It would be very interesting to extend the result of
Theorem~\ref{main-theorem} (or some parts of it) to complete, separated
topological rings $\R$ with a countable base of neighborhoods of zero
consisting of open \emph{right} ideals.
 Is such a generalization true?
 See~\cite[Questions~9.8, 11.2, and~13.3]{Pflcc} for further discussion.
\end{qst}

\end{document}